\documentclass{amsart}

\usepackage{pdfsync, amsmath, graphicx, amsthm, amssymb, url, booktabs, bm}

\usepackage{tikz} 
\usetikzlibrary{matrix,decorations.pathmorphing, positioning}
\usetikzlibrary{arrows,automata,backgrounds,petri}
\usetikzlibrary{shapes.multipart}
\usetikzlibrary{calc, trees}
\usetikzlibrary{cd}

\usepackage{enumitem} 
\usepackage[alphabetic, initials]{amsrefs}
\usepackage{mathrsfs} % Use \mathscr
\usepackage[normalem]{ulem}

\theoremstyle{plain}
\newtheorem{theorem}{Theorem}[section]
\newtheorem{corollary}[theorem]{Corollary}

\newtheorem{lemma}[theorem]{Lemma}
\newtheorem*{theorem*}{Theorem}
\newtheorem*{lemma*}{Lemma}
\newtheorem*{proposition*}{Proposition}
\newtheorem*{corollary*}{Corollary}

\theoremstyle{definition}

\newtheorem*{definition*}{Definition}
\newtheorem*{example*}{Example}
\newtheorem*{remark*}{Remark}

\newcommand{\LLL}{\mathscr{L}}

\newcommand{\I}{\mathrm{I}}
\newcommand{\II}{\mathrm{II}}
\newcommand{\III}{\mathrm{III}}

\DeclareMathOperator{\HH}{Ht}
\DeclareMathOperator{\dd}{d}

\title[Diophantine approximation on circles and spheres]{Intrinsic Diophantine approximation on circles and spheres}

\author{Byungchul Cha}
\address{Muhlenberg College, 2400 Chew st, Allentown, PA, 18104, USA}
\email{cha@muhlenberg.edu}

\author{Dong Han Kim}
\address{Department of Mathematics Education, Dongguk University - Seoul, 30 Pildong-ro 1-gil, Jung-gu, Seoul, 04620 Korea}
\email{kim2010@dgu.ac.kr}

\thanks{Research supported by the National Research Foundation of Korea (NRF-2018R1A2B6001624).}



\begin{document}
\begin{abstract}
We study Lagrange spectra arising from intrinsic Diophantine approximation of circles and spheres. 
More precisely, we consider three circles embedded in $\mathbb{R}^2$ or $\mathbb{R}^3$ and three spheres embedded in $\mathbb{R}^3$ or $\mathbb{R}^4$. 
We present a unified framework to connect the Lagrange spectra of these six spaces with the spectra of $\mathbb{R}$ and $\mathbb{C}$. 
%As a corollary, combining this with prior work of Asmus L.~Schmidt on the spectra of $\mathbb{R}$ and $\mathbb{C}$, we can characterize, for each of the six spectra, the smallest accumulation point and the initial discrete part leading up to it completely.
Thanks to prior work of Asmus L.~Schmidt on the spectra of $\mathbb{R}$ and $\mathbb{C}$, we obtain as a corollary, for each of the six spectra, the smallest accumulation point and the initial discrete part leading up to it completely.
\end{abstract}
%\begin{abstract}
%Let $\LLL(S)$ be the Lagrange spectrum arising from the intrinsic Diophantine approximation of the spheres and the circles  by their rational points. 
%We study Diophantine approximation properties of the circles and spheres and show that three spheres have Diophantine equivalence with imaginary quadratic number field of $\mathbb Q(\sqrt{-1})$, $\mathbb Q(\sqrt{-2})$ and $\mathbb Q(\sqrt{-3})$.
%\end{abstract}

\maketitle

%\tableofcontents

\section{Introduction}\label{sec:introduction}
In the classical Diophantine approximation, we study approximations of irrational numbers by rational numbers. 
For any irrational number $\alpha$,  there exists an $L \ge 1$ such that the inequality 
\[
\left| \alpha - \frac pq \right| < \frac{1}{L q^2}
\]
admits infinitely many integer solutions $p$ and $q$ that are relatively prime. 
The \emph{Lagrange number} $L(\alpha)$ of $\alpha$ is defined to be the supremum of such $L$'s.
Equivalently, we have  
\[
L(\alpha) =\limsup_{p/q \in \mathbb Q} \left( q^2 \left| \alpha - \frac pq \right| \right)^{-1}.
\]
The \emph{Lagrange spectrum} $\LLL$ is the set of Lagrange numbers (that are finite) of all irrational numbers.
Many properties of $\LLL$ are known. 
For example, $\sqrt 5$ is the smallest value of $\LLL$, called the \emph{Hurwitz constant}, and $3$ is the smallest limit point of $\LLL$ in $\mathbb{R}$.
See \cite{Aig13}, \cite{Bom07} and \cite{CF89} for details.

On the other hand, structures of Lagrange spectra arising from \emph{intrinsic} Diophantine approximation are less known. 
After introducing a few notations, 
we briefly review some existing results on intrinsic Diophantine approximation, focusing on the case of unit spheres $S^n$ in $\mathbb{R}^{n+1}$. 
%This will provide a general context for our results.
Let $(\mathcal{X}, \dd)$ be a complete metric space.
%and let $\mathcal{Y}$ be a closed subset of $\mathcal{X}$.
Assume that 
%$\mathcal{Y}$ is contained in the closure of a countable subset 
$\mathcal{Z}$ is a countable dense subset of $\mathcal{X}$.
In addition, we assume that there is a \emph{height function} $\HH_{\mathcal{Z}}: \mathcal{Z} \longrightarrow \mathbb{R}_{\ge0}$.
Given the data
$(\mathcal{X}, \dd,
\mathcal{Z}, \HH_{\mathcal{Z}})$,
we define \emph{the Lagrange number} $L_{(\mathcal{X}, \dd, \mathcal{Z}, \HH_{\mathcal{Z}})}(P)$ \emph{of} $P\in \mathcal{X} \setminus \mathcal Z$ to be
\[
L_{(\mathcal{X}, \dd, \mathcal{Z}, \HH_{\mathcal{Z}})}(P) 
= \limsup_{Z\in \mathcal{Z}}\frac1{\HH_{\mathcal{Z}}(Z) \dd(P, Z)}.
\]
Also, we define \emph{the Lagrange spectrum} to be
%{\color{blue}
\[
\LLL(\mathcal{X}, \dd, \mathcal{Z}, \HH_{\mathcal{Z}}) = 
\{ L_{(\mathcal{X}, \dd, \mathcal{Z}, \HH_{\mathcal{Z}})}(P) 
\mid P \in \mathcal{X} \setminus \mathcal Z, \quad 
L_{(\mathcal{X}, \dd, \mathcal{Z}, \HH_{\mathcal{Z}})}(P) > 0 \}.
\]
When the choices of $\dd$ and $\HH_{\mathcal{Z}}$ are clear from the context, 
we will write $L_{(\mathcal{X}, \mathcal{Z})}(P)$ 
instead of $L_{(\mathcal{X}, \dd, \mathcal{Z}, \HH_{\mathcal{Z}})}(P)$.
Furthermore, we may 
simply write
%drop even the subscript $(\mathcal{X}, \mathcal{Z})$ altogether and simply write 
$L(P)$ 
%for $L_{(\mathcal{X}, \dd, \mathcal{Z}, \HH_{\mathcal{Z}})}(P)$ 
whenever the data $(\mathcal{X}, \dd, \mathcal{Z}, \HH_{\mathcal{Z}})$ 
is implicitly understood
and thus there is no danger of confusion.
Similarly, we will write $\LLL(\mathcal{X}, \mathcal{Z})$ when the context makes clear the choices of $\dd$ and $\HH_{\mathcal{Z}}$.

Using these notations, if we let $\mathcal{X} =\mathbb{R}$, equipped with the usual Euclidean distance $\dd$, 
and $\mathcal{Z} = \mathbb{Q}$ with the height function on $\mathbb{Q}$ being $\HH_{\mathbb{Q}}(p/q) = |q|^2$,
the case $(\mathcal{X}, \mathcal{Z}) = (\mathbb{R}, \mathbb{Q})$
corresponds to
%{\color{brown} \textbf{Old version:}
%With these notations, 
%the choice $(\mathcal{X}, \mathcal{Z}) = (\mathbb{R}, \mathbb{Q})$,
%with the height function on $\mathbb{Q}$ being $\HH_{\mathbb{Q}}(p/q) = |q|^2$,
%would give rise to 
%}
the classical Lagrange spectrum studied by Markoff in \cite{Mar79} and \cite{Mar80}.
%See \cite{PP17} for more on this in a much broader context than what we consider in this paper.

Particularly relevant to the present article is the intrinsic Diophantine approximation of $n$-spheres
$(\mathcal{X}, \mathcal{Z}) = (S^n, S^n \cap \mathbb{Q}^{n+1})$.
%with $H(Z) = q$ for $Z = (\frac{p_1}{q}, \dots ,\frac{p_n}{q})$
%This was studied in \cite{KM15} and \cite{FKMS}. 
Here, 
\[
S^n = \{ (x_1, x_2, \dots, x_{n+1} )\in \mathbb R^{n+1} \, | \, x_1^2 + x_2^2+ \dots + x_{n+1}^2 = 1\}
\]
is the unit $n$-sphere in $\mathbb{R}^{n+1}$.
And the distance $\dd$ on $S^n$ is the Euclidean distance inherited from the ambient space $\mathbb{R}^{n+1}$.
Also, we define the height function 
on $S^n \cap \mathbb{Q}^{n+1}$
to be
\begin{equation}
\HH_{S^n \cap \mathbb{Q}^{n+1}}(\mathbf{p}/q) =|q|
\label{eq:Height_in_Sn}
\end{equation}
with $\mathbf{p}\in\mathbb{Z}^{n+1}$ 
\emph{primitive}, 
meaning that all coefficients of $\mathbf{p}$ have no common divisor $>1$.  

With respect to this data
$(S^n, \dd, S^n \cap \mathbb{Q}^{n + 1},\HH_{S^n \cap \mathbb{Q}^{n + 1}})$,
it appears that
Kopetzky was the first person to determine the initial discrete part of the spectrum 
$\LLL(S^1, S^1\cap\mathbb{Q}^2)$
of the 1-sphere.
He obtained his results in \cite{Kop80} and \cite{Kop85} by relying on prior results \cite{Sch75a} and \cite{Sch75b} of Asmus L.~Schmidt. 
%\todo{We need to quote a paper. Also, find out exactly what distance function and height function he used. I think his distance is the same with us, so we don't need to mention it.}
In \cite{CK23}, the authors of the present paper investigated $\LLL(S^1, S^1\cap\mathbb{Q}^2)$ independently of Kopetzky's results and determined the structure of the initial discrete part of 
$\LLL(S^1, S^1\cap\mathbb{Q}^2)$
more explicitly by using a different set of tools.

For $n \ge 2$, much less is known about 
$\LLL(S^n, S^n\cap\mathbb{Q}^{n+1})$.
Moshchevitin \cite{Mos16} found that the minimum of 
$\LLL(S^2, \dd, S^n\cap\mathbb{Q}^{3}, \HH_{S^2\cap\mathbb{Q}^{3}})$ 
%$\LLL(S^2, S^2\cap\mathbb{Q}^{3})$
is greater than or equal to $\frac 12 \sqrt{\frac{\pi}{3}}$.
In \cite{KM15}, Kleinbock and Merrill made an important contribution to the study of implicit Diophantine approximation of the $n$-sphere for general $n\ge 2$.
See also \cite{FKMS}.
We should note here that 
Kleinbock and Merrill use the sup norm $\dd_{\sup}(\mathbf{x} , \mathbf{y}) := \sup_{i=0}^n \lvert x_i - y_i\rvert$ on $S^n$, not the Euclidean distance $\dd$,
while they use the same height function as in \eqref{eq:Height_in_Sn}.
One of their results in \cite{KM15} says that, for each $n\ge 1$,
%$\LLL(S^n, \dd_{\sup}, S^n\cap\mathbb{Q}^{n+1}, \HH_{S^n\cap\mathbb{Q}^{n+1}})$ 
%is bounded away from zero.
%In particular, 
%for each $n\ge 1$, 
there exists $c_n > 0$ such that
\[
\LLL(S^n, \dd_{\sup}, S^n\cap\mathbb{Q}^{n+1}, \HH_{S^n\cap\mathbb{Q}^{n+1}})
\subseteq (c_n , \infty).
%\subsetneq (0, \infty).
\]

When we combine \cite{Sch75a} with Corollary~\ref{cor:equality_lagrange_spectra_S2},
our result shows that the minimum of 
$\LLL(S^2, \dd, S^2\cap\mathbb{Q}^{3}, \HH_{S^2\cap\mathbb{Q}^{3}})$ 
is $ \sqrt{\frac{3}{2}}$,
which improves 
Moshchevitin's lower bound $\frac12\sqrt{\frac{\pi}{3}}$.
%mentioned above.
Note that the case $\mathcal{X} = \mathbf{S}_{\I}^2$ in Corollary~\ref{cor:equality_lagrange_spectra_S2} is the same 
as $(S^2, S^2 \cap \mathbb{Q}^3)$ above.
Additionally, it is easy to see that, for any $\mathbf{x}, \mathbf{y}\in S^n$, 
\[
\frac1{\sqrt{n + 1}} \dd(\mathbf{x}, \mathbf{y})
\le
\dd_{\sup}(\mathbf{x}, \mathbf{y})
\le
\dd(\mathbf{x}, \mathbf{y}).
\]
From this, we see that
\begin{multline}
L_{(S^n, \dd, S^n\cap\mathbb{Q}^{n+1}, \HH_{S^n\cap\mathbb{Q}^{n+1}})}(P) 
\\
\le
L_{(S^n, \dd_{\sup}, S^n\cap\mathbb{Q}^{n+1}, \HH_{S^n\cap\mathbb{Q}^{n+1}})} (P)
\\
\le 
\sqrt{n +  1} \cdot
L_{(S^n, \dd, S^n\cap\mathbb{Q}^{n+1}, \HH_{S^n\cap\mathbb{Q}^{n+1}})} (P)
\end{multline}
for $P \in S^n \setminus \mathbb{Q}^{n+1}$.
So we conclude 
\begin{equation}
\sqrt{
\frac32
}
\le
\inf
\LLL(S^2, \dd_{\sup}, S^2\cap\mathbb{Q}^{3}, \HH_{S^2\cap\mathbb{Q}^{3}})
\le 
\frac3{\sqrt2}.
\end{equation}
%}
%{\color{brown} \textit{I think that the discussion about the Eisenstein case is not that relevant and they may be dropped entirely.}
%Obviously, 
%$\LLL(\mathcal{X}, \mathcal{Z})$
%is sensitive 
%to the choice of $\mathcal{Z}$ (and the height function on it), even when $\mathcal{X}$ is the same.
%For instance, let
%\[
%\begin{cases}
%\mathcal{X} = \{ z \in \mathbb{C} \mid | z | = 1 \}, \\
%\mathcal{Z} = \{ z \in \mathcal{X} \mid  z \in \mathbb{Q}(\sqrt{-3}) \}.
%\end{cases}
%\]
%Then the main result in \cite{CCGW} gives the description of the initial discrete part of $\LLL(\mathcal{X}, \mathcal{Z})$, which is different from that of $\LLL(S^1, S^1 \cap \mathbb{Q}^2)$.
%See \cite{CCGW} for the definition of height function in this case.
%}

%{\color{blue}
The plan for the paper is as follows.
We aim to study Lagrange spectra of certain spaces $\mathcal{X}$ listed in Table~\ref{tab:notations}, all of which are homeomorphic to $S^1$ or $S^2$.
Note that $\mathbf{S}^1_\I = S^1$ and $\mathbf{S}^2_\I = S^2$ in the above discussions.
We will construct maps from
either $\mathbb{R}$ or $\mathbb{C}$ 
to $\mathcal{X}$ 
via stereographic projections. %}
To be precise, we introduce the six spaces $(\mathcal{X}, \mathcal{Z})$
that are listed in Table~\ref{tab:notations}.
\begin{table}
\[
\begin{array}{lc}
\toprule
\mathcal{X} & \mathcal{Z} \\
\midrule
\mathbf{S}^1_\I  = \{ (x_1, x_2) \in \mathbb{R}^2 \mid x_1^2 + x_2^2 = 1  \}
& \mathbf{S}^1_\I \cap \mathbb{Q}^2  \\
\mathbf{S}^1_\II  = \{ (x_1, x_2) \in \mathbb{R}^2 \mid x_1^2 + x_2^2 = 2\}
& \mathbf{S}^1_\II \cap \mathbb{Q}^2  \\
% \{ (x_0, x_1, x_2)\in \mathbb{R}^3 \mid x_0 + x_1 + x_2 = 1 \} & 
\mathbf{S}^1_\III  = \{ (x_0, x_1, x_2) \in W \mid x_0^2 + x_1^2 + x_2^2 = 1  \}
& \mathbf{S}^1_\III \cap \mathbb{Q}^3  \\
\quad \text{where} \quad W = \{ (x_0, x_1, x_2) \in \mathbb{R}^3 \mid x_0 + x_1 + x_2 = 1 \} &  \\
\midrule
\mathbf{S}^2_\I  = \{ (x_1, x_2, x_3) \in \mathbb{R}^3 \mid x_1^2 + x_2^2 + x_3^2 = 1  \}
& \mathbf{S}^2_\I \cap \mathbb{Q}^3  \\
\mathbf{S}^2_\II  = \{ (x_1, x_2, x_3) \in \mathbb{R}^3 \mid x_1^2 + x_2^2 + x_3^2 = 2\}
& \mathbf{S}^2_\II \cap \mathbb{Q}^3  \\
\mathbf{S}^2_\III  = \{ (x_0, x_1, x_2, x_3) \in W \mid x_0^2 + x_1^2 + x_2^2 + x_3^2 = 1  \}
& \mathbf{S}^2_\III \cap \mathbb{Q}^4  \\
\quad \text{where} \quad W = \{ (x_0, x_1, x_2, x_3) \in \mathbb{R}^4 \mid x_0 + x_1 + x_2 + x_3 = 1 \} &  \\
\bottomrule
\end{array}
\]
\caption{Definitions of the six spaces $(\mathcal{X}, \mathcal{Z})$. %{\color{blue} 
In all cases, the distance $\dd$ on $\mathcal{X}$ is the Euclidean distance in the ambient spaces $\mathbb{R}^l$ of $\mathcal{X}$. And $\HH_{\mathcal{Z}}$ is given by $\HH_{\mathcal{Z}}(\mathbf{p}/q) = |q|$ with primitive $\mathbf{p}\in \mathbb{Z}^l$.
}%}
\label{tab:notations}
\end{table}
%{\color{red} 
%Three spaces $\mathbf S^1_\I, \mathbf S^1_\II, \mathbf S^1_\III$ are metrically isomorphic to $S^1$ with different rational point set $\mathcal Z$ and
%$\mathbf S^2_\I, \mathbf S^2_\II, \mathbf S^2_\III$ are metrically isomorphic to $S^2$.
%Note that $\mathbf S^1_\I = S^1$ and $\mathbf S^2_\I = S^2$.}
In each of the six cases, the distance function $\dd$ on $\mathcal{X}$ will be the usual Euclidean distance inherited from its ambient Euclidean space $\mathbb{R}^{n+1}$ or, in the case of $\mathbf{S}_\III^n$, $\mathbb{R}^{n+2}$.
Also, we define the set $\mathcal{Z}$ of \emph{rational points} to be 
$\mathcal{Z} = \mathcal{X} \cap \mathbb{Q}^l$, where $l = n+1$ or $n+2$.
%those $\mathbf{p} = (p_0, \dots, p_l)\in \mathcal{X}$ with $p_i \in \mathbb{Q}$ for all $i = 0, \dots, l$.
The height function on $\mathcal{Z}$ is defined in the same way as in \eqref{eq:Height_in_Sn}.
For each of the three cases $\mathcal{X} = 
\mathbf{S}^1_\I,
\mathbf{S}^1_\II,$ 
$\mathbf{S}^1_\III$,  we will construct 
%later in \S\ref{Sec:one_dimension},
a map $\Phi: \mathbb{R} \longrightarrow \mathcal{X}$.
When $\mathcal{X} = 
\mathbf{S}^2_\I,
\mathbf{S}^2_\II,$ or
$\mathbf{S}^2_\III$,  the map will be 
$\Phi: \mathbb{C} \longrightarrow \mathcal{X}$.
%In \S\ref{Sec:Property_of_maps},
Then we present a common set of conditions 
%{\color{blue}
($\Phi$-i) and ($\Phi$-ii)%} 
(see Lemma~\ref{lem:main_lemma})
for all the above maps $\Phi$ to satisfy.
%{\color{blue}
In \S\ref{sec:summary_of_heights},
we define the height functions on certain sets $K$ of rational points of $\mathbb{R}$ and of $\mathbb{C}$.
When the conditions 
($\Phi$-i) and ($\Phi$-ii)
are %}
satisfied, the Lagrange numbers of $\xi$ and $\Phi(\xi)$ are the same, up to a fixed constant factor.
%Because we utilize many spaces whose sets of rational points have varying definitions of height functions, we summarize the definitions of height functions in \S\ref{sec:summary_of_heights} for easy reference.
In \S\ref{Sec:one_dimension}, we give the definitions of $\Phi$ for each of the six cases and prove that 
($\Phi$-i) and ($\Phi$-ii)
are satisfied for these six maps.

%As far as we can tell, the Hurwitz bounds for $\LLL(S^n)$ with $n > 2$ seem to be unknown at the moment. 
%We say two system $(\mathcal X_1, \mathcal Z_1, H_1)$ and $(\mathcal X_2, \mathcal Z_2, H_2)$ are Diophantine equivalent if there exists a bijective map $\Phi: \mathcal X_1 \to \mathcal X_2$ such that 
%$$\Phi(\mathcal Z_1) = \mathcal Z_2 \ \text{  and } \ L_1(t) = L_2(\Phi(t)) \text{ for all } t \in \mathcal X_1 \setminus \mathcal Z_1.$$
%It is achieved if for any $z \in \mathcal Z_1$
%$$
%\lim_{x \to z} \frac{ \dd_2 (\Phi(x),\Phi(z)) }{\dd_1(x,z)}
%=\frac{\HH_1(z)}{\HH_2(\Phi(z))}
%$$
%and
%$$
%\Phi'(y) = \lim_{x \to y} \frac{\dd_2 (\Phi(x),\Phi(y))}{\dd_1 (x,y)}
%$$
%is defined and continuous.
%{\color{blue}
We note here that $(\mathbb{C}, K)$, even when $K = \mathbb{Q}(\sqrt{-1})$, is not the same as $(\mathbb{R}^2, \mathbb{Q}^2)$ and therefore our results do not establish any relation between
the Lagrange spectrum of $S^2$ and that of $\mathbb{R}^2$. 
For this reason, the authors do not know if one can use a similar method to study $(S^n, S^n \cap \mathbb{Q}^{n+1})$ for $n\ge 3$.
%}
%{\color{red} We remark that the Diophantine approximation on the complex plane $\mathbb C$ is different from the Diophantine approximation on $\mathbb R^2$. For $n \ge 3$, we do not know the corresponding $(\mathcal X, \mathcal Z)$ for $S^n$.}

Now, we state our main theorem for $S^1$, that is, for the cases $\mathcal{X} = 
\mathbf{S}^1_\I$, 
$\mathbf{S}^1_{\II}$,
$\mathbf{S}^1_{\III}$. 
\begin{theorem}\label{thm:new_main1}
%{\color{blue}
Consider $(\mathbb{R}, K)$ where $K = \sqrt2\mathbb{Q}$ for the statements \textrm{(a)} and \textrm{(b)}
and $K = \mathbb{Q}$ for \textrm{(c)}.
The distance $\dd$ on $\mathbb{R}$ is the usual Euclidean distance and the definitions of height functions on $K$ are given in \S\ref{sec:summary_of_heights}.
%}
\begin{enumerate}[font=\upshape, label=(\alph*)]
\item
There exists a continuous bijection 
$
\Phi^1_{\I}: \mathbb{R} \longrightarrow \mathbf{S}^1_{\I} \setminus \{ \mathbf{n} \}
$
for some $\mathbf{n} \in 
\mathbf{S}^1_{\I} \cap \mathbb{Q}^2$ 
such that 
$\Phi^1_{\I}$ maps $\sqrt2\mathbb{Q}$ onto 
$\mathbf{S}^1_{\I} \cap \mathbb{Q}^2 \setminus \{ \mathbf{n} \}$ 
and
\[
L_{(\mathbb{R}, \sqrt2\mathbb{Q})}(\xi) = \sqrt 2 L_{(\mathbf{S}^1_{\I}, \mathbf{S}^1_{\I}\cap \mathbb{Q}^2)}(\Phi (\xi))
\]
for every $\xi \in \mathbb{R} \setminus \sqrt2\mathbb{Q}$.
\item
There exists a continuous bijection
$
\Phi^1_{\II}: \mathbb{R} \longrightarrow \mathbf{S}^1_{\II} \setminus \{ \mathbf{n} \}
$
for some $\mathbf{n} \in 
\mathbf{S}^1_{\II} \cap \mathbb{Q}^2$ 
such that 
$\Phi^1_{\II}$ maps $\sqrt2\mathbb{Q}$ onto 
$\mathbf{S}^1_{\II} \cap \mathbb{Q}^2 \setminus \{ \mathbf{n} \}$ 
and
\[
L_{(\mathbb{R}, \sqrt2\mathbb{Q})}(\xi) = 2 L_{(\mathbf{S}^1_{\II}, \mathbf{S}^1_{\II}\cap \mathbb{Q}^2)}(\Phi (\xi))
\]
for every $\xi \in \mathbb{R} \setminus \sqrt2\mathbb{Q}$.
\item
There exists a continuous bijection
$
\Phi^1_{\III}: \mathbb{R} \longrightarrow \mathbf{S}^1_{\III} \setminus \{ \mathbf{n} \}
$
for some $\mathbf{n} \in 
\mathbf{S}^1_{\III} \cap \mathbb{Q}^3$ 
such that 
$\Phi^1_{\III}$ maps $\mathbb{Q}$ onto 
$\mathbf{S}^1_{\III} \cap \mathbb{Q}^3 \setminus \{ \mathbf{n} \}$ 
and
\[
L_{(\mathbb{R}, \mathbb{Q})}(\xi) = \sqrt2 L_{(\mathbf{S}^1_{\III}, \mathbf{S}^1_{\III}\cap \mathbb{Q}^3)}(\Phi (\xi))
\]
for every $\xi \in \mathbb{R} \setminus \mathbb{Q}$.
\end{enumerate}
\end{theorem}
\begin{corollary}\label{cor:equality_lagrange_spectra_S1}
For $S^1$, we have
%\begin{enumerate}[font=\upshape, label=(\alph*)]
%\item $\displaystyle
%\LLL(\mathbf S^1_{\I},  \mathbf S^1_{\I} \cap \mathbb{Q}^2)
%=
%\frac1{\sqrt2}
%\LLL(\mathbb{R}, \sqrt2\mathbb{Q}),
%$
%\item $\displaystyle
%\LLL(\mathbf S^1_{\II},  \mathbf S^1_{\II} \cap \mathbb{Q}^2)
%=
%\frac1{2}
%\LLL(\mathbb{R}, \sqrt2\mathbb{Q}),
%$
%\item $\displaystyle
%\LLL(\mathbf S^1_{\III},  \mathbf S^1_{\III} \cap \mathbb{Q}^2)
%=
%\frac1{\sqrt2}
%\LLL(\mathbb{R}, \mathbb{Q}).
%$
%\end{enumerate}
\begin{align*}
\LLL(\mathbf S^1_{\I},  \mathbf S^1_{\I} \cap \mathbb{Q}^2)
&=
\frac1{\sqrt2}
\LLL(\mathbb{R}, \sqrt2\mathbb{Q}),\\
\LLL(\mathbf S^1_{\II},  \mathbf S^1_{\II} \cap \mathbb{Q}^2)
&=
\frac1{2}
\LLL(\mathbb{R}, \sqrt2\mathbb{Q}),\\
\LLL(\mathbf S^1_{\III},  \mathbf S^1_{\III} \cap \mathbb{Q}^3)
&=
\frac1{\sqrt2}
\LLL(\mathbb{R}, \mathbb{Q}).
\end{align*}
\end{corollary}
The usefulness of Corollary~\ref{cor:equality_lagrange_spectra_S1} comes from the fact that the spectra
$
\LLL(\mathbb{R}, \sqrt2\mathbb{Q})
$
and
$
\LLL(\mathbb{R}, \mathbb{Q})
$
have been studied by many authors and we already know quite a bit about their structure.
%It is not difficult to see that 
%our Lagrange number in $(\mathbb{R}, \sqrt2\mathbb{Q})$
%given by the height function in \eqref{eq:modified_height}
%results in the same notion of A.~Schmidt's \emph{$C$-minimum} in \cite{Sch75b}
%and therefore we conclude that
%$\LLL(\mathbb{R}, \sqrt2\mathbb{Q})$ is the same as 
%Schmidt's spectrum in \cite{Sch75b}.
For example, it is possible to relate A.~Schmidt's prior result in \cite{Sch76} with the space $(\mathbb{R}, \sqrt2\mathbb{Q})$. 
See \cite{KS} %Remark~\ref{rem:sqrt2Q_height}
for more discussion on this.
As a result, we deduce that the smallest accumulation point of 
$\LLL(\mathbf S^1_{\I},  \mathbf S^1_{\I} \cap \mathbb{Q}^2)$
is $2$
and 
\begin{multline*}
\LLL(\mathbf S^1_{\I},  \mathbf S^1_{\I} \cap \mathbb{Q}^2) \cap (0, 2) 
= \\
\left\{
\sqrt{4 - \frac1{x^2}}\Big| x= 1, 5, 11, 29, \dots
\right\} 
\cup
\left\{
\sqrt{4 - \frac2{y^2}}\Big| y= 1, 3, 11, 17, \dots
\right\}.
\end{multline*}
Here, $x$ and $y$ 
arises from an integer solution $(x, y_1, y_2)$ (with $y = y_1$ or $y = y_2$) satisfying the Diophantine equation
\begin{equation} \label{eq:Markoff_equation_2}
2x^2 + y_1^2 + y_2^2 = 4xy_1y_2.
\end{equation}
This gives a complete description of the initial discrete part of 
$\LLL(\mathbf S^1_{\I},  \mathbf S^1_{\I} \cap \mathbb{Q}^2)$.
One can make similar statements for
$\LLL(\mathbf S^1_{\II},  \mathbf S^1_{\II} \cap \mathbb{Q}^2)$.

%And, of course the above corollary directly relates $\LLL(\mathbf S^1_{\III},  \mathbf S^1_{\III} \cap \mathbb{Q}^2)$ to Markoff's classical description of 

Next, we present our results for $S^2$, that is, for the cases $\mathcal{X} = \mathbf{S}^2_{\I}$, $\mathbf{S}^2_{\II}$,  $\mathbf{S}^2_{\III}$. 
%Again, the definitions of height functions used in the Lagrange numbers are summarized in \S\ref{sec:summary_of_heights}.
%Let $K$ be an imaginary quadratic field with class number 1. In other words, $K$
%is an imaginary quadratic field with fundamental discriminant $d < 0$ where $d$ is one of the seven values:
%\[
%d = -3, -4, -7, -8, -11, -19, -43, -163.
%\]
%(We will use three of these, $d = -4, -8, -3$, in Theorem~\ref{thm:new_main2} below.)
%We fix an embedding of $K$ into $\mathbb{C}$, so that $K$ is a countable dense subset of $\mathbb{C}$.
%%Then its ring $\mathcal{O}_K$ of integers is given by $\mathcal{O}_K = \mathbb{Z} + \mathbb{Z}\omega$ with
%%\[
%%\omega =
%%\begin{cases}
%%\tfrac12 \sqrt d & \text{ if } d \equiv 0 \bmod 4, \\
%%\tfrac12(1 + \sqrt d) & \text{ if } d \equiv 1 \bmod 4.
%%\end{cases}
%%\]
%To define the height function $\HH_K$ on $K$, we first write an element $z$ of $K$ as a \emph{reduced} fraction $z = \alpha/\beta$, that is, $\alpha$ and $\beta$ are integers of $K$ which generate the maximal order of $K$ over $\mathbb{Z}$.
%Then we define
%\begin{equation}\label{eq:height_class_number_one_imag_quad}
%\HH_K(z) = 
%\HH_{K} \left( \frac{\alpha}{\beta} \right) = \beta \beta^{\sigma} = | \beta |^2.
%\end{equation}
%Here, $\beta^{\sigma}$ is the complex conjugate of $\beta$ in $\mathbb{C}$.
%Notice that, due to the class number one condition on $K$, the expression $z = \alpha/\beta$ is unique up to a unit of $K$ and therefore $\HH_K(z)$ is well-defined.
%

\begin{theorem}\label{thm:new_main2}
%{\color{blue}
Consider $(\mathbb{C}, K)$ where 
\[
K = 
\begin{cases}
\mathbb{Q}(\sqrt{-1}) &\text{ for (a)}, \\
\mathbb{Q}(\sqrt{-2}) &\text{ for (b)}, \\
\mathbb{Q}(\sqrt{-3}) &\text{ for (c)}.
\end{cases}
\]
The distance $\dd$ on $\mathbb{C}$ is the usual Euclidean distance and the definitions of height functions on $K$ are given in \S\ref{sec:summary_of_heights}.
%}
\begin{enumerate}[font=\upshape, label=(\alph*)]
\item
There exists a continuous bijection 
\[
\Phi^2_{\I}: \mathbb{C} \longrightarrow \mathbf{S}^2_{\I} \setminus \{ \mathbf{n} \}
\]
for some $\mathbf{n} \in 
\mathbf{S}^2_{\I} \cap \mathbb{Q}^3$ 
such that 
$\Phi^2_{\I}$ maps $\mathbb{Q}(\sqrt{-1})$ onto 
$\mathbf{S}^2_{\I} \cap \mathbb{Q}^3 \setminus \{ \mathbf{n} \}$ 
and
\[
L_{(\mathbb{C}, \mathbb{Q}(\sqrt{-1}))}(\xi) = \sqrt 2 L_{(\mathbf{S}^2_{\I}, \mathbf{S}^2_{\I}\cap\mathbb{Q}^3)}(\Phi (\xi))
\]
for every $\xi \in \mathbb{C} \setminus \mathbb{Q}(\sqrt{-1})$.
\item
There exists a continuous bijection
\[
\Phi^2_{\II}: \mathbb{C} \longrightarrow \mathbf{S}^2_{\II} \setminus \{ \mathbf{n} \}
\]
for some $\mathbf{n} \in 
\mathbf{S}^2_{\II} \cap \mathbb{Q}^3$ 
such that 
$\Phi^2_{\II}$ maps $\mathbb{Q}(\sqrt{-2})$ onto 
$\mathbf{S}^2_{\II} \cap \mathbb{Q}^3 \setminus \{ \mathbf{n} \}$ 
and
\[
L_{(\mathbb{C}, \mathbb{Q}(\sqrt{-2}))}(\xi) = \sqrt2 L_{(\mathbf{S}^2_{\II}, \mathbf{S}^2_{\II}\cap\mathbb{Q}^3)}(\Phi (\xi))
\]
for every $\xi \in \mathbb{C} \setminus \mathbb{Q}(\sqrt{-2})$.
\item
There exists a continuous bijection
\[
\Phi^2_{\III}: \mathbb{C} \longrightarrow \mathbf{S}^2_{\III} \setminus \{ \mathbf{n} \}
\]
for some $\mathbf{n} \in 
\mathbf{S}^2_{\III} \cap \mathbb{Q}^4$ 
such that 
$\Phi^2_{\III}$ maps $\mathbb{Q}(\sqrt{-3})$ onto 
$\mathbf{S}^2_{\III} \cap \mathbb{Q}^4 \setminus \{ \mathbf{n} \}$ 
and
\[
L_{(\mathbb{C}, \mathbb{Q}(\sqrt{-3}))}(\xi) = \sqrt2 L_{(\mathbf{S}^2_{\III}, \mathbf{S}^2_{\III}\cap\mathbb{Q}^4)}(\Phi (\xi))
\]
for every $\xi \in \mathbb{C} \setminus \mathbb{Q}(\sqrt{-3})$.
\end{enumerate}
\end{theorem}

\begin{corollary}\label{cor:equality_lagrange_spectra_S2}
For $S^2$, we have
\begin{align*}
\LLL(\mathbf S^2_{\I}, \mathbf{S}^2_{\I}\cap\mathbb{Q}^3) 
= \frac{1}{\sqrt 2} \LLL\left(\mathbb C, \mathbb Q(\sqrt{-1}) \right), \\
\LLL(\mathbf S^2_{\II}, \mathbf{S}^2_{\II}\cap\mathbb{Q}^3) = \frac{1}{\sqrt 2} \LLL\left(\mathbb C, \mathbb Q(\sqrt{-2}) \right), \\
\LLL(\mathbf S^2_{\III}, \mathbf{S}^2_{\III}\cap\mathbb{Q}^4) = \frac{1}{\sqrt 2} 
\LLL\left(\mathbb C, \mathbb Q(\sqrt{-3}) \right).
\end{align*}
\end{corollary}
As before, this corollary allows us to use the known theorems of A.~Schmidt \cite{Sch75a}, \cite{Sch11}, \cite{Sch83} about the spectra
$\LLL\left(\mathbb C, \mathbb Q(\sqrt{-1}) \right)$,
$\LLL\left(\mathbb C, \mathbb Q(\sqrt{-2}) \right)$,
and
$\LLL\left(\mathbb C, \mathbb Q(\sqrt{-3}) \right)$,
and to obtain
complete descriptions of the initial discrete parts of the three spectra of $S^2$ in the corollary.
For instance, 
the smallest limit point of $\LLL(\mathbf S^2_{\I}, \mathbf{S}^2_{\I}\cap\mathbb{Q}^3)$
is $\sqrt2$
and
its initial discrete part is
\[
\LLL(\mathbf S^2_{\I}, \mathbf{S}^2_{\I}\cap\mathbb{Q}^3)
\cap (0, \sqrt2)
=
\left\{
\sqrt{
2 - \frac1{2x^2}
}
\Big|
x = 1, 5, 11, 29, \dots
\right\}
\cup
\left\{
\sqrt{
\frac3{10}
\sqrt{41}
}
\right\}
.
\]
Here, $x$ is the first integer in the triple $(x, y_1, y_2)$ which satisfies the Diophantine equation \eqref{eq:Markoff_equation_2}.
As for $\LLL(\mathbf S^2_{\II}, \mathbf{S}^2_{\II}\cap\mathbb{Q}^3)$,
the minimum value (so-called the \emph{Hurwitz's bound}) 
is 1 and 
its smallest limit point 
is (cf.~\cite{Sch11})
\[
\left(
\frac{
4(82 662 667 + 115 77720 \sqrt{47})
}{
405 186 721
}
\right)^{1/2}.
\]
Also, the minimum value of $\LLL(\mathbf S^2_{\III}, \mathbf{S}^2_{\III}\cap\mathbb{Q}^4)$
is $(13/4)^{1/4}$
and its smallest limit point
is (cf.~\cite{Sch83})
\[
\left(
\frac{
14 + 8 \sqrt3
}{
13}
\right)^{1/2}.
\]

In our last theorem below (Theorem~\ref{thm:horocycle_new}), we show that the six maps $\Phi$ in Theorems~\ref{thm:new_main1} and \ref{thm:new_main2} can be extended to hyperbolic spaces (either a hyperbolic plane or a hyperbolic 3-space) as isometries.
To explain, 
let $\mathbf{B}$ be the open ball whose boundary is $\mathbf{S}$.
Also we write
$\mathbb H^2 = \{ (x,t) \in \mathbb R^2 \, | \, t >0\}$ 
and
$\mathbb H^3 = \{ (z,t) \in \mathbb C \times \mathbb R \, | \, t >0\}$, 
the hyperbolic plane and the hyperbolic 3-space, respectively, having $\hat{\mathbb{R}} = \mathbb R \cup \{\infty \}$ or $\hat{\mathbb{C}}= \mathbb C \cup \{\infty \}$ as its boundary.
%Also we let $\mathbb{F}$ be the domain of $\Phi$, that is,
%\begin{equation}
%    \mathbb{F} = 
%    \begin{cases}
%    \mathbb{R} & \text{ if } \mathbf{S} = 
%    \mathbf{S}^1_{\I},
%    \mathbf{S}^1_{\II},
%    \text{ or }
%    \mathbf{S}^1_{\III},
%    \\
%    \mathbb{C} & \text{ if } \mathbf{S} = 
%    \mathbf{S}^2_{\I},
%    \mathbf{S}^2_{\II},
%    \text{ or }
%    \mathbf{S}^2_{\III}.
%    \end{cases}
%\end{equation}
In \S\ref{sec:proof_of_extension}, we will construct a map $\bar{\Phi}$ (for each of the six cases)
from either $\mathbb{H}^2 \cup \hat{\mathbb{R}}$  or $\mathbb{H}^3 \cup \hat{\mathbb{C}}$ to  $\overline{\mathbf{B}}$,
which coincides with $\Phi$ on the boundary of its domain.
Furthermore, we will show that $\bar{\Phi}$ satisfies the following properties:

\begin{table}
\[
\begin{array}{ccc || ccc}
\toprule
\mathcal{X} & R & C & \mathcal{X} & R & C \\
\midrule
\mathbf{S}^1_\I  & 1 & \sqrt2 &
\mathbf{S}^2_\I  & 1 & \sqrt2 \\
\mathbf{S}^1_\II  & \sqrt2 & 2 &
\mathbf{S}^2_\II  & \sqrt2 & \sqrt2 \\
\mathbf{S}^1_\III  & \sqrt{2/3} & \sqrt2 &
\mathbf{S}^2_\III  & \sqrt3/2 & \sqrt2\\
\bottomrule
\end{array}
\]
\caption{The radii %constants 
$R$ and dilation factors $C$ for each of the six cases of $\mathcal{X}$.}
\label{tab:Constants_R_C}
\end{table}
\begin{theorem}\label{thm:horocycle_new}
The map $\bar{\Phi}$ is an isometry from either $\mathbb{H}^2 \cup \hat{\mathbb{R}}$ or $\mathbb{H}^3 \cup \hat{\mathbb{C}}$ to $\overline{\mathbf{B}}$,
such that the restriction of $\bar{\Phi}$ to the boundary of its domain is the same as $\Phi$.
In addition, for the radii $R$ of $\mathbf S$ and dilation factors $C$ for $\Phi$ %are defined 
defined in Section 2 (see Table~\ref{tab:Constants_R_C}),
%Then
$\bar{\Phi}$ maps a horosphere 
(in $\mathbb{H}^2$ or $\mathbb{H}^3$) 
based at a rational point $z\in K$ 
(cf.~\eqref{def:field_K})
with radius $1/(2\HH_{K}(z))$ to a horosphere in $\mathbf{B}$ based at $\Phi(z)$ with radius
\begin{equation}
\label{eq:rho}
\rho = \frac{R}{1 + (2 R/C) \cdot \HH_{\mathbf{S}\cap \mathbb{Q}^l}(\Phi(z))}.
\end{equation}
Furthermore, any two horospheres based at two distinct rational points $\mathbf{z}$ and $\mathbf{z}'$ in $\mathbf S$ are either tangent or disjoint. 
\end{theorem}
Geometric meanings of the dilation factors $C$ will become clearer once we introduce all related objects in \S\ref{Sec:Property_of_maps}.
To be brief, %$R$ is the radius of the sphere $\mathbf{S}$ and 
$C$ is a dilation factor by which either $\mathbb{R}$ or $\mathbb{C}$ is enlarged, before it is mapped onto $\mathbf{S}$.
Also, notice that $C$ is the ratio between two Lagrange numbers in Theorems~\ref{thm:new_main1} and \ref{thm:new_main2}.
See Lemma~\ref{lem:main_lemma}, especially the condition ($\Phi$-i).
The  horosphere (in $\mathbb{H}^2$ or $\mathbb{H}^3$) 
based at a rational point $z\in K$ 
with radius $1/(2\HH_{K}(z))$
is called the Ford horosphere.
%(See Figure~\ref{fig:RQ}).
By Theorem~\ref{thm:horocycle_new}, the geometric picture of the Ford horoshpere can be applied to the Diophantine approximation on $\mathbf S$.

\section{Set up and some lemmas}\label{Sec:Property_of_maps}
Let $n$ be a positive integer and let $l \ge n + 1$. 
We fix an $(n+1)$-dimensional plane $W$ inside of $\mathbb{R}^l$, which is defined over $\mathbb{Q}$. 
This plane $W$, which will play the role of an ambient space, is equipped with the Euclidean distance $\dd(\mathbf{x}, \mathbf{y}) = | \mathbf{x} - \mathbf{y}|$ inherited from $\mathbb{R}^l$. 

Later in this paper, we will let $n = 1$ for $S^1$ and $n = 2$ for $S^2$. 
Furthermore, for $\mathbf S^{n}_{\I}$ and $\mathbf S^{n}_{\II}$,
we will let $l = n + 1$ and $W = \mathbb{R}^{n+1}$.
For $\mathbf S^{n}_{\III}$, we will let $l = n + 2$ and $W$ be the 
$(n+1)$-dimensional hyperplane in $\mathbb{R}^{n+2}$ given by
\[
W = \{ (x_0, x_2, \dots, x_{n+1}) \in \mathbb{R}^{n+2} \mid x_0 + \cdots + x_{n+1} = 1 \}.
\]
%Some of our discussion in this section, however, will be valid for any $(n+1)$-dimensional plane $W \subset \mathbb{R}^l$ defined over $\mathbb{Q}$.

We fix the data $(\mathbf{S}, \mathbf{n}, \mathbf{P})$, satisfying the following conditions: %all of which are either subsets or points in $W$:
\begin{itemize}
\item $\mathbf{S}$ is an $n$-dimensional sphere (inside of $W$) centered at a point $\mathbf{c} \in W$
with radius $R>0$ possessing a rational point $\mathbf{n}$, that is, $\mathbf{n} \in \mathbf{S}\cap \mathbb{Q}^l$,
%\todo{I put in the extra condition that $\mathbf{n}$ be in $\mathbb{Q}^l$.}
\item $\mathbf{P}$ is an $n$-dimensional plane (inside of $W$) not containing $\mathbf{n}$ and perpendicular to $\mathbf{n} - \mathbf{c}$ (with respect to the usual dot product in $\mathbb{R}^l$). 
We denote by $D$ the (shortest) distance between $\mathbf{n}$ and $\mathbf{P}$.
\item The product $RD$ is a rational number.
%\todo{I think this is an essential assumption.
%(In all six cases, this is true.) 
%Without this, the rational points in $\mathbb{R}$ or $\mathbb{C}$ may not map to rational points in $\mathbf{S}$.
%The condition $RD\in\mathbb{Q}$ will ensure that $\Psi$ is defined over $\mathbb{Q}$, ensuring that every rational point in one space gets mapped to another under $\Phi$.}
\end{itemize}
Additionally, we let $\tilde{\mathbf S}\subset W$ be the sphere centered at $\mathbf n$ with radius $\sqrt{2RD}$
and write $\hat{W} 
= W \cup \{ \infty \}$. 
Then we define 
$\Psi: \hat{W} \longrightarrow \hat{W}$ to be the \emph{reflection in}
$\tilde{\mathbf{S}}$
\emph{of}  $\hat{W}$ (cf.~\S3.1 in \cite{Beardon}).
%Denote by 
%$\mathrm{Ref}_{\tilde{\mathbf S}}$
\begin{figure}
\begin{tikzpicture}[scale=1.3]
\draw[thick] (-2, 0) -- (5, 0) node[right] {$\mathbf P$};
\draw[->] (0, -.6) -- (0, 3.5);
% R = 1, d = 1/2
\draw (0,1.5) circle (1.732); 
%node [above left] {$\mathbf n$}; 
\draw[thick] (0,.5) circle (1); %node [Right] {$\mathbf S$}; 

\draw (0,1.5) -- (4,-.5);

\node at (-1,1.2) {$\mathbf S$};
\node at (1.6,2.8) {$\tilde{\mathbf S}$};
\node at (-.4,1) {$R$};
\node at (.25,.7) {$D$};

\draw[<->] (0,.5) -- (-.8,1.1);
\draw[<->] (.1,1.5) -- (.1,0);

\draw[fill] (0,1.5) circle (0.05) node [above right] {$\mathbf n$};
\draw[fill] (0, .5) circle (0.05) node [below left] {$\mathbf c$};
\draw[fill] (3,0) circle (0.05) node [below] {$\mathbf x$};
%	\draw[fill] (0,0) circle (0.05) node [above left] {$\mathbf o$};

\draw[fill] (.8,1.1) circle (0.05) node [above right] {$\Psi(\mathbf x)$};

%\draw[red] (.8,1.1) arc (36.87:36.87+360:0.2); 
%\draw[red] (3,.5) circle (0.5);
%\draw[red] (3,-.5) circle (0.5); 

\end{tikzpicture}
\caption{The stereographic projection $\Psi(\mathbf{x})$.
}\label{projection}
\end{figure}
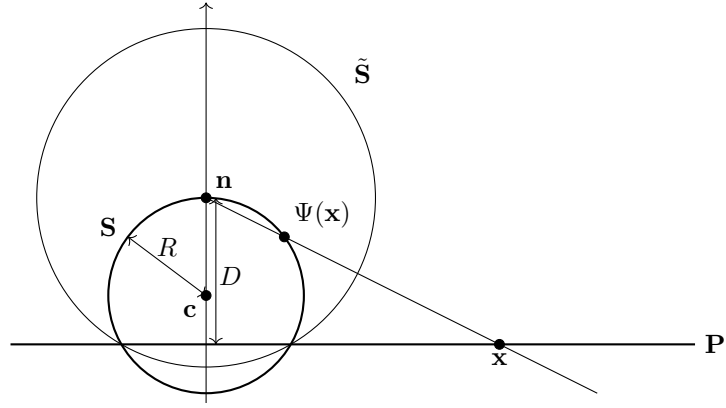
From a general formula of reflections (see for instance (3.1.1) of \cite{Beardon}), we have
\begin{equation}\label{eq:reflection_s_tilde}
\Psi(\mathbf x) 
%\mathrm{Ref}_{\tilde{\mathbf S}}  (\mathbf x) 
= \mathbf n + \frac{2RD(\mathbf x - \mathbf n)}{|\mathbf x - \mathbf n|^2}.
\end{equation}
%Here, $\| \cdot \|$ is the Euclidean length given by the standard inner product in $\mathbb{R}^{\ell}$.
%Also,
%$\mathrm{Ref}_{\tilde{\mathbf{S}}}%\Psi
%(\mathbf{n}) = \infty$
%and
%$\mathrm{Ref}_{\tilde{\mathbf{S}}}%\Psi
%(\infty) = \mathbf{n}$.
The restriction of $\Psi$ to $\mathbf{S}$ 
gives a \emph{stereographic projection of} $\mathbf{S}$  \emph{at} $\mathbf{n}$.
More precisely, the point $\Psi(\mathbf{x})$ coincides with  the (unique) point of $\mathbf{P}$ such that the three points $\mathbf x, \Psi(\mathbf{x}),\mathbf n$ lie on the same line.
On the other hand, when $\Psi$ is restricted to $\mathbf{P}$, it induces a one-to-one correspondence $\Psi: \mathbf{P}\longrightarrow \mathbf{S} \setminus \{\mathbf{n}\}$.
Recall that $RD$ is assumed to be a rational number and consequently $\Psi$ is defined over $\mathbb{Q}$ (see \eqref{eq:reflection_s_tilde} above).
Therefore, $\Psi$ also maps $\mathbf{P}\cap \mathbb{Q}^l$ bijectively onto $(\mathbf{S}\setminus \{ \mathbf{n} \})\cap \mathbb{Q}^l$.
The following lemma is a straightforward consequence of \eqref{eq:reflection_s_tilde} so we will omit the proof.
\begin{lemma}[the chordal metric on $\mathbf{S}$]\label{lem:chordal_metric}
For $\mathbf x, \mathbf y \in \mathbf P$, 
%We have for $\mathbf x, \mathbf y \in \mathbf P$
%By a direct calculation we have the chordal metric on $\mathbf P$ as follows
\[
%d(\mathbf o + \mathbf v, \mathbf o + \mathbf w) = 
\left| \Psi (\mathbf x) - \Psi (\mathbf y) \right| 
= \frac{2RD | \mathbf x - \mathbf y |}{|\mathbf x - \mathbf n | | \mathbf y - \mathbf n |}.
\]
\end{lemma}

%\subsection{Main Lemma}\label{subsection:plan}
Let $\mathbb{F}$ be 
\begin{equation}\label{def:field_F}
\mathbb{F} = 
\begin{cases}
\mathbb{R} & \text{ if } n = 1,\\
\mathbb{C} & \text{ if } n = 2.
\end{cases}
\end{equation}
Also, let $K$ be a countable dense subset of $\mathbb{F}$, which is equipped with a height function $\HH_K$.
Later, we will choose $K$ to be
\begin{equation}\label{def:field_K}
K = 
\begin{cases}
\sqrt2 \mathbb{Q} & \text{ for } \mathbf{S}^1_{\I},\\
\sqrt2 \mathbb{Q} & \text{ for } \mathbf{S}^1_{\II},\\
\mathbb{Q} & \text{ for } \mathbf{S}^1_{\III},
\end{cases}
\quad
\text{and}
\quad
\begin{cases}
\mathbb{Q}(\sqrt{-1})& \text{ for } \mathbf{S}^2_{\I},\\
\mathbb{Q}(\sqrt{-2})& \text{ for } \mathbf{S}^2_{\II},\\
\mathbb{Q}(\sqrt{-3})& \text{ for } \mathbf{S}^2_{\III}.
\end{cases}
\end{equation}
Suppose that there is a continuous bijection
\begin{equation}\label{def:map_varphi}
\varphi: \mathbb{F} \longrightarrow \mathbf{P} 
\end{equation}
such that $\varphi(K) = \mathbf{P}\cap \mathbb{Q}^l$.
Finally, we define
\begin{equation}\label{def:map_Phi}
\Phi = \Psi %\mathrm{Ref}_{\tilde{\mathbf{S}}} 
\circ \varphi: \mathbb{F} \longrightarrow \mathbf{S}\setminus\{ \mathbf{n} \}.
\end{equation}
%The most important ingredient of our proof for Theorem~\ref{thm:main} is to show that 
%$\varphi$ satisfies the following properties:

\begin{lemma}[Main Lemma]\label{lem:main_lemma}
Suppose that the map $\Phi$ satisfies the two conditions:
\begin{enumerate}[font=\upshape, label=($\Phi$-\roman*)]
%\item[($\Phi$-i)]
\item
There exists a positive constant $C>0$ such that 
\[
| \varphi(x_1) - \varphi(x_2) | = C |x_1 - x_2|
\]
for all $x_1, x_2 \in \mathbb{F}$.
%\item[($\Phi$-ii)]
\item
For any $z\in K$,
\[
\frac{ \HH_{\mathbf S \cap \mathbb{Q}^l}(\Phi(z)) }{ \HH_{K}(z) }
=
\frac{ | \varphi(z)  - \mathbf n |^2 }{2RD}
\]
\end{enumerate}
Then
\[
L_{(\mathbb{F}, K)}(\xi) = C \cdot L_{(\mathbf{S}, \mathbf{S}\cap \mathbb{Q}^l)}(\Phi(\xi))
\]
for every $\xi\in \mathbb{F}\setminus K$.
\end{lemma}
\begin{proof}
Let $\xi\in \mathbb{F}\setminus K$ and $z \in K$.
From Lemma~\ref{lem:chordal_metric}, we have
\[
| \Phi(\xi) - \Phi(z) |
= 
\frac{2RD | \varphi(\xi) - \varphi(z) |}{|\varphi(\xi)  - \mathbf n | | \varphi(z) - \mathbf n |}.
\]
Combine this with ($\Phi$-i) and ($\Phi$-ii) to obtain
\begin{align*}
\frac{
\HH_{\mathbf S\cap \mathbb{Q}^l}(\Phi(z)) 
| \Phi(\xi) - \Phi(z) |
}{
\HH_{K}(z) | \xi - z |
}
&= 
\frac{
\HH_{\mathbf S\cap \mathbb{Q}^l}(\Phi(z)) 
}{
\HH_{K}(z)
}
\frac{|\varphi(\xi) - \varphi(z)|}{ | \xi - z |}
\frac{
2RD
}{|\varphi(\xi)  - \mathbf n | | \varphi(z) - \mathbf n |}
\\
&=
C \frac{| \varphi(z) - \mathbf n |}{| \varphi(\xi)  - \mathbf n |}.
\end{align*}
This proves the lemma.
%\[
%L_{\mathbb{C}, K}(\xi) = C L_{\mathbf{S}}(\Phi(\xi)),
%\]
%which proves the lemma.
\end{proof}

\section{Summary of height functions}
\label{sec:summary_of_heights}
%In this section, we summarize the height functions used in this paper. 
\subsection{The height function on $(\mathbb{R}, K)$ when $\mathbf{S} = \mathbf{S}^1_{\I}, \mathbf{S}^1_{\II}, \mathbf{S}^1_{\III}$}\label{sec:height_function_for_S1}
Recall that the domain of our map $\Phi$ in this case is $\mathbb{F} = \mathbb{R}$.
We choose the set $K$ of its rational points to be
\[
K = 
\begin{cases}
\sqrt2 \mathbb{Q} & \text{ if } \mathbf{S} = \mathbf{S}^1_{\I}, \\
\sqrt2 \mathbb{Q} & \text{ if } \mathbf{S} = \mathbf{S}^1_{\II}, \\
\mathbb{Q} & \text{ if } \mathbf{S} = \mathbf{S}^1_{\III}.
\end{cases}
\]
For the last case  $K = \mathbb{Q}$, the height function on $\mathbb{Q}$ is the usual one, that is, $\HH_{\mathbb{Q}}(\tfrac pq) = q^2$ for a reduced fraction $\frac pq$.
When $K = \sqrt2\mathbb{Q}$, we define the height function on $\sqrt2 \mathbb{Q}$ as follows.
Let $r\in\sqrt2\mathbb{Q}$. 
Write $r = \sqrt2 \cdot p/q_1$ with coprime integers $p$ and $q_1$.
If $q_1$ is odd, then we let $q = q_1$. 
If $q_1$ is even, then ($p$ must be odd and) we let $q = q_1/2$.
As a result, we can express $r$ as 
\begin{equation}
\label{eq:r_as_fraction}
r =
\begin{cases}
\tfrac{\sqrt2 p}q & \text{ with $q$ odd, or}\\
\tfrac{p}{\sqrt2 q} & \text{ with $p$ odd}
\end{cases}
\end{equation}
for coprime integers $p$ and $q$ in a unique way (up to the signs of $p$ and $q$).
We define
\begin{equation}\label{eq:modified_height}
\HH_{\sqrt 2 \mathbb Q}(r) =
\begin{cases}
\HH_{\sqrt 2 \mathbb Q}
\left( \frac{\sqrt 2 p}{q} \right)
= q^2, 
& \text{ if } r = \frac{\sqrt2 p}{q} \text{ with $q$ odd}, \\
\HH_{\sqrt 2 \mathbb Q} \left( \frac{p}{\sqrt 2 q} \right)
= 2q^2,
& \text{ if } r = \frac{p}{\sqrt2 q} \text{ with $p$ odd}.
\end{cases}
\end{equation}
To put it in another way, suppose that $x$ is the numerator and $y$ is the denominator in the fraction in \eqref{eq:r_as_fraction}, so that $r = x/y$ in both cases.
The height of $r$ is then 
\begin{equation}
\label{eq:Height_sqrt2Q_fractional_expression}
\HH_{\sqrt2\mathbb{Q}}(r) 
=
\HH_{\sqrt2\mathbb{Q}}(\tfrac xy)  = y^2.
\end{equation}
Notice that $x^2$, $y^2$ and $\sqrt 2 xy$ are all integers satisfying
\begin{equation}\label{coprime}
\gcd \left( x^2, y^2, \sqrt 2 xy \right)  = 1.
\end{equation}
For more detailed discussions on the Diophantine approximation in $(\mathbb R, \sqrt 2 \mathbb Q)$, consult \cite{KS}.
%\begin{remark}\label{rem:sqrt2Q_height}
%This and that.
%\end{remark}

\subsection{The height function on $(\mathbb{C}, K)$ when $\mathbf{S} = \mathbf{S}^2_{\I}, \mathbf{S}^2_{\II}, \mathbf{S}^2_{\III}$}
The domain of $\Phi$ in this case is $\mathbb{F} = \mathbb{C}$ and the set $K$ of its rational points is chosen to be
\[
K = 
\begin{cases}
\mathbb{Q}(\sqrt{-1}) & \text{ if } \mathbf{S} = \mathbf{S}^2_{\I}, \\
\mathbb{Q}(\sqrt{-2}) & \text{ if } \mathbf{S} = \mathbf{S}^2_{\II}, \\
\mathbb{Q}(\sqrt{-3}) & \text{ if } \mathbf{S} = \mathbf{S}^2_{\III}.
\end{cases}
\]
In this case, we denote by $\mathcal{O}_K$ the ring of integers of $K$ and write each $r\in K$ as $r = \alpha/\beta$ where $\alpha$ and $\beta$ are elements of $\mathcal{O}_K$ such that the ideal generated by $\alpha$ and $\beta$ is the entire $\mathcal{O}_K$. 
Then we define
\begin{equation}\label{eq:height_class_number_one_imag_quad_2}
\HH_K(z) = 
\HH_{K} \left( \frac{\alpha}{\beta} \right) = \beta \beta^{\sigma} = | \beta |^2.
\end{equation}
Here, $\beta^{\sigma}$ is the complex conjugate of $\beta$ in $\mathbb{C}$.
Since the class number of $\mathcal{O}_K$ is one for the three choices of $K$ above, the expression $z = \alpha/\beta$ is unique up to a unit of $\mathcal{O}_K$ and therefore $\HH_K(z)$ is well-defined.

\subsection{The height function on $(\mathbf{S},\mathbf{S}\cap \mathbb{Q}^l)$}
In all six cases of $\mathbf{S}$, the rational points of $\mathbf{S}$ are defined by $\mathbf{S} \cap \mathbb{Q}^l$ (where $l = n + 1$ or $n + 2$). 
And we utilize the standard height function on $\mathbb{Q}^l$. 
That is, for each $r\in \mathbf{S} \cap \mathbb{Q}^l$, we write $r$ as $r = \mathbf{p}/q$ with $\mathbf{p} = (p_1, \dots, p_l) \in \mathbb{Z}^l$ and $q\in \mathbb{Z}$ such that
\begin{equation}\label{eq:gcd_condition_p_and_q}
\gcd(p_1, \dots, p_l, q) = 1
\end{equation}
and define
\begin{equation}\label{eq:height_in_S_for_six_spheres}
\HH_{\mathbf{S} \cap \mathbb{Q}^l}(r) = q.
\end{equation}
Note that the integers $p_1, \dots, p_l, q$ satisfy
\[
p_1^2 + \cdots + p_l^2 = k q^2
\]
for $k=1$ or $k=2$. 
Therefore the gcd condition \eqref{eq:gcd_condition_p_and_q} is equivalent to $\gcd(p_1, \dots, p_l) = 1$.
In other words, the definition \eqref{eq:height_in_S_for_six_spheres} is equivalent to the one used by Kleinbock and Merrill \eqref{eq:Height_in_Sn}.

\section{Proof of Theorem~\ref{thm:new_main1}; the case of $S^1$}\label{Sec:one_dimension}

\subsection{The case $\mathbf{S} = \mathbf{S}^1_{\I}$}\label{sec:CaseS1_I}
Let $W = \mathbb{R}^2$ and define
\[
\begin{cases}
\mathbf{S} = \{ (x_1, x_2) \in \mathbb{R}^2 \mid x_1^2 + x_2^2 = 1 \},\\
\mathbf{n} = (0, 1), \\ 
\mathbf{P} = \{ (x_1, 0) \in \mathbb{R}^2 \mid x_1 \in \mathbb{R} \}.
\end{cases}
\]
This yields
\begin{equation}
R = D = 1, %\quad \mathbf{o} = (0, 0, 0)
\end{equation}
so that the conditions in \S\ref{Sec:Property_of_maps} are satisfied.
Also, the formula \eqref{eq:reflection_s_tilde} becomes
\begin{equation}
\label{eq:Psi_case_A1}
\Psi(x_1, x_2) =
\frac{
(2x_1, x_1^2 + x_2^2 - 1)}{
x_1^2 + x_2^2 - 2x_2 + 1}.
\end{equation}

Next, we define
$\varphi : \mathbb{R} \longrightarrow \mathbf{P}$ to be
\begin{equation}
\label{eq:def_phi_A1}
\varphi(t) = \left( \sqrt 2 t -1, 0 \right).
\end{equation}
Note that $\varphi$ maps $\sqrt2 \mathbb{Q}$ to $\mathbf{P}\cap \mathbb{Q}^2$ bijectively.
Also, this gives
\[
| \varphi(x_1) -\varphi(x_2) | = \sqrt 2 | x_1 - x_2 |.
\]
so $\varphi$ satisfies the condition ($\Phi$-i) in Lemma~\ref{lem:main_lemma} with $C = \sqrt 2$.

To verify the condition ($\Phi$-ii),
we compute 
$\HH_{\sqrt2 \mathbb{Q}}(r)$ and $\HH_{\mathbf S\cap \mathbb{Q}^2}(\Phi(r))$
for $r \in \sqrt 2 \mathbb Q$.
Using the notations in \S\ref{sec:height_function_for_S1}, we write 
$r=x/y$, so that $\HH_{\sqrt2\mathbb{Q}}(r) = y^2$ (cf.~\eqref{eq:Height_sqrt2Q_fractional_expression}).
To compute $\HH_{\mathbf S\cap \mathbb{Q}^2}(\Phi(r))$, we use \eqref{eq:Psi_case_A1} to get
\begin{align*}
\Phi(r) &= \Phi(\varphi(r)) =
\frac1{r^2 - \sqrt2 r + 1}
\left(\sqrt2 r - 1, r^2 - \sqrt2 r\right) \\
&= \frac{\left(\sqrt 2 xy - y^2, x^2 -\sqrt 2 xy \right)}{x^2 -\sqrt 2 xy + y^2}.
\end{align*}
We verify the gcd condition \eqref{eq:gcd_condition_p_and_q} for $\Phi(r)$ using \eqref{coprime}:
\begin{equation*}
\gcd \left( \sqrt 2 xy -y^2, x^2 -\sqrt 2 xy,x^2 - \sqrt 2 xy + y^2 \right)  = 
\gcd \left( x^2, y^2, \sqrt 2 xy \right)  = 1.
\end{equation*}
This gives 
\begin{equation}
\label{eq:Ht_case_A1}
\HH_{\mathbf S\cap \mathbb{Q}^2}(\Phi(r)) = x^2 -\sqrt 2 xy + y^2.
\end{equation}
So, 
\[
\frac{
\HH_{\mathbf S\cap \mathbb{Q}^2}(\Phi(r))
}{
   \HH_{\sqrt2 \mathbb{Q}}(r)
}
=
\frac{
x^2 - \sqrt2 xy + y^2}{
y^2
}
=r^2 - \sqrt2 r + 1.
\]
On the other hand, 
\begin{equation}\label{eq:factor_case_A1}
\frac{| \varphi(r) - \mathbf{n} |^2}{2RD} =
\frac{| ( \sqrt 2 r  - 1, -1)|^2}{2}
=r^2 -\sqrt 2 r + 1.
\end{equation}
This proves ($\Phi$-ii) in this case.

\subsection{The case $\mathbf{S} = \mathbf{S}^1_{\II}$}\label{sec:CaseS1_II}
As in \S\ref{sec:CaseS1_I}, we let $W = \mathbb{R}^2$.
We define
\[
\begin{cases}
\mathbf{S} = \{ (x_1, x_2) \in \mathbb{R}^2 \mid x_1^2 + x_2^2 = 2 \},\\
\mathbf{n} = (1, 1), \\ 
\mathbf{P} = \{ (x_1, x_2) \in \mathbb{R}^2 \mid x_1 + x_2 = 0 \}.
\end{cases}
\]
This yields
\begin{equation}
R = \sqrt2, \quad D = \sqrt2 
%\quad \text{and} \quad \mathbf{o} = (0, \tfrac12, \tfrac12).
\end{equation}
and 
\begin{equation}
\label{eq:Psi_case_B1}
\Psi(x_1, x_2) =
\frac{
(x_1^2 + x_2^2 +2x_1 - 2x_2 -2, x_1^2 + x_2^2 - 2x_1 + 2x_2 -2)}{
x_1^2 + x_2^2 - 2x_1 - 2x_2 + 2}.
\end{equation}

In this case, we define $\varphi: \mathbb{R} \longrightarrow \mathbf{P}$ to be
\begin{equation}
\label{eq:def_phi_B1}
\varphi(t) = \left(\sqrt 2 t - 1, 1 - \sqrt 2 t\right).
\end{equation}
Then
%To prove ($\Phi$-i) in this case, we let 
we get
\[ 
| \varphi(t_1) - \varphi(t_2) |^2
= \left| ( \sqrt 2 (t_1 - t_2), \sqrt 2 (t_2-t_1)) \right| ^2 
= 4 (t_1-t_2)^2,
\]
which proves ($\Phi$-i) with $C = 2$.

For ($\Phi$-ii), let $r\in\sqrt2\mathbb{Q}$ and write $r = x/y$ as before, so that $\HH_{\sqrt2\mathbb{Q}}(r) = y^2$. 
To compute $\HH_{\mathbf S\cap \mathbb{Q}^2}(\Phi(r))$ in this case, we use \eqref{eq:Psi_case_B1} and \eqref{eq:def_phi_B1} to get
\begin{align*}
\Phi(r) &= 
\Psi\left(\varphi(r)\right)
= \Psi\left(\sqrt 2 r - 1, 1 - \sqrt 2 r  \right)
\\
&=
\frac{(r^2 - 1, r^2 - 2\sqrt2r + 1)}{r^2 - \sqrt2 r + 1}
=
\frac{(x^2 - y^2, x^2 - 2\sqrt 2 xy +y^2)}
{x^2 - \sqrt 2 xy + y^2}.
\end{align*}
The gcd condition \eqref{eq:gcd_condition_p_and_q} is verified by \eqref{coprime}, that is,
\begin{equation*}
\gcd \left(
x^2 - y^2, 
x^2 - 2\sqrt 2 xy +y^2, 
x^2 - \sqrt 2 xy + y^2 \right)
=
\gcd \left( x^2, y^2, \sqrt 2 xy \right)  = 1.
\end{equation*}
Therefore we have
\begin{equation}
\label{eq:Ht_case_B1}
\HH_{\mathbf S\cap \mathbb{Q}^2}(\Phi(r))
=
x^2 - \sqrt 2 xy + y^2.
\end{equation}
On the other hand,
%\[
%\varphi(\tfrac xy) = 
%\left(\frac{\sqrt 2 x}y -1, 1 -\frac{\sqrt 2 x}y\right).
%\]
\begin{equation}\label{eq:factor_case_B1}
\frac{| \varphi(r ) - \mathbf{n} |^2}{2RD} =
\frac{ 
(\sqrt 2 r- 2)^2 + 2r^2
}{4}
=r^2 - \sqrt2r + 1.
\end{equation}
%To compute $\HH(\Phi(\frac xy))$, we compute 
Combining this with \eqref{eq:Ht_case_B1}, we obtain
($\Phi$-ii) in this case.

\subsection{The case $\mathbf{S} = \mathbf{S}^1_{\III}$}\label{sec:CaseS1_III}
In this case, we let the ambient space $W$ be 
\[
W 
= \{ (x_0, x_1, x_2) \in \mathbb{R}^3 \mid
x_0 + x_1 + x_2 = 1
\}
\]
and define
\[
\begin{cases}
\mathbf{S} = \left\{ (x_0, x_1, x_2) \in W
\, | \, x_0^2 + x_1^2 + x_2^2 = 1 \right\},\\
\mathbf{n} = (0, 0, 1), \\ 
\mathbf{P} = \{ (x_0, x_1, 0) \in W %\mathbb{A}^3 
\mid x_0 + x_1 = 1 \}.
\end{cases}
\]
This gives
\begin{equation}
R = \sqrt{\frac 23}, \quad D = \sqrt{\frac 32} %\quad \text{and} \quad \mathbf{o} = (\tfrac13, \tfrac13, \tfrac13, 0).
\end{equation}
and %When $(x_1, x_2, x_3, 0) \in \mathbf{P}$, 
\begin{equation}
\label{eq:Psi_case_C1}
\Psi(x_0, x_1, x_2) =
\frac{
(2x_0, 2x_1, x_0^2 + x_1^2 + x_2^2 - 1)}{
x_0^2 + x_1^2 + x_2^2 - 2x_2 + 1}.
\end{equation}
Note that $\Psi(\mathbf x) \in W$ for $\mathbf x \in W$.
In this case, 
We define $\varphi : \mathbb{R} \longrightarrow \mathbf{P}$ to be
\begin{equation}
\label{eq:def_phi_C1}
\varphi(t) = (t, 1-t, 0).
\end{equation}
The property ($\Phi$-i) is easily seen to be satisfied with $C = \sqrt 2$ because
$$
| \varphi(t_1) - \varphi(t_2) |^2 = 2(t_1 - t_2)^2.
$$
For each $r\in\mathbb{Q}$, we write $r = p/q$ as a reduced fraction, so that $\HH_{\mathbb{Q}}(r) = q^2$.
Also, 
To compute $\HH_{\mathbf{S}\cap\mathbb{Q}^3}(\Phi(r))$, we see that
\[
\Phi (r)
%\Psi\left( \varphi (r)\right)
=\Psi\left(
r, 1-r, 0
\right)
=
\frac{(r, 1- r,r^2 - r)}{r^2 - r + 1}
=
\frac{(pq, q^2 - pq,p^2 - pq)}{p^2 + q^2 - pq}.
\]
Since
\[
\gcd \left(pq, q^2 - pq,p^2 - pq, %{\color{red} 
p^2 + q^2 - pq %}
\right)
= \gcd \left( p^2, q^2, pq \right)  = 1
\]
the gcd condition \eqref{eq:gcd_condition_p_and_q} is satisfied in this case and 
we have
\begin{equation}
\label{eq:Ht_case_C1}
\HH_{\mathbf S \cap \mathbb{Q}^3}\left(\Phi(r)\right)
= p^2  + q^2 - pq.
\end{equation}
%\[
%\varphi(\tfrac pq) = 
%(\tfrac pq, 1 - \tfrac pq, 0).
%\]
Finally,
\begin{equation}
\label{eq:factor_case_C1}
\frac{| \varphi(r) - \mathbf{n} |^2}{2RD} =
\frac{r^2 + (1-r)^2 + (-1)^2}{2}
=\frac{p^2 -pq + q^2}{q^2}.
\end{equation}
Therefore, 
\eqref{eq:Ht_case_C1}
and 
\eqref{eq:factor_case_C1}
give
($\Phi$-ii) in this case. 

\section{Proof of Theorem~\ref{thm:new_main2}; the case of $S^2$}\label{Sec:two_dimension}
To prove Theorem~\ref{thm:new_main2}, we will construct a map 
$\varphi: \mathbb{C} \longrightarrow \mathbf{P}$
for each of the cases 
$\mathbf{S} = \mathbf{S}^2_{\I}$,
$\mathbf{S}^2_{\II}$,
$\mathbf{S}^2_{\III}$.
As explained in \S\ref{sec:introduction},
the set of rational points $K$ of $\mathbb{C}$ will be chosen to be
\[
K = 
\begin{cases}
\mathbb{Q}(\sqrt{-1}) & \text{ for the case }\mathbf{S} = \mathbf{S}^2_{\I},\\
\mathbb{Q}(\sqrt{-2}) & \text{ for the case }\mathbf{S} = \mathbf{S}^2_{\II},\\
\mathbb{Q}(\sqrt{-3}) & \text{ for the case }\mathbf{S} = \mathbf{S}^2_{\III}.
\end{cases}
\]
It will be convenient to choose an integral basis over $\mathbb{Z}$ for the ring $\mathcal{O}_K$ of integers of $K$ as 
$\mathcal{O}_K = \mathbb{Z} + \mathbb{Z} \omega_K$ where
\[
\omega_K =
\begin{cases}
\mathrm{i} = \sqrt{-1} & \text{ for the case }\mathbf{S} = \mathbf{S}^2_{\I},\\
\sqrt{-2} &  \text{ for the case }\mathbf{S} = \mathbf{S}^2_{\II}, \\
(-1 + \sqrt{-3})/2 &  \text{ for the case }\mathbf{S} = \mathbf{S}^2_{\III}.
\end{cases}
\]

%\begin{lemma}[Older version]
%\label{lem:old_reduced_fraction_in_Ok}
%The notations are as above.
%Suppose $p, q\in \mathcal{O}_K$ with $\gcd(p, q) = 1$ and $q \neq 0$.
%(Recall that $K$ is assumed to be of class number one, so $\mathcal{O}_K$ is a UFD.)
%Then
%$\gcd(
%|p|^2,
%|q|^2
%) = 1$.
%Furthermore, let  $a, b, c\in \mathbb{Z}$ be given by
%\[
%c = 
%|q|^2
%%\mathrm{Norm}_{K/\mathbb{Q}}(q) 
%\quad
%\text{and}
%\quad
%a + b\omega_K = p\overline{q}.
%\]
%Then
%\[
%\frac pq = 
%\frac{ a + b\omega_K}c
%\]
%and
%\[
%\mathrm{Norm}_{K/\mathbb{Q}}(p).
%|p|^2
%= \frac{%\mathrm{Norm}_{K/\mathbb{Q}}
%|a + b\omega_K|^2
%}{c}.
%\]
%\end{lemma}

\begin{lemma}
\label{lem:reduced_fraction_in_Ok}
The notations are as above.
Let $r$ be a nonzero element $K$ and write $r = \alpha/\beta$
where $\alpha$ and $\beta$ are elements of $\mathcal{O}_K$ which generate $\mathcal{O}_K$.
Let $a, b, c$ be the (rational) integers given by
\[
a + b\omega_K = \alpha\beta^{\sigma}
\quad
\text{and}
\quad
c = |\beta|^2 = \beta\beta^{\sigma}.
\]
Here, $\sigma$ is the nontrivial automorphism of $K/\mathbb{Q}$ (the complex conjugation).
Then we have
\begin{enumerate}[font=\upshape, label=(\alph*)]
\item $r = (a + b\omega_K)/c,$
\item $\HH_{K}(r) = c$,
\item $c$ divides $|a + b\omega_K|^2$ (in $\mathbb{Z})$, 
\item $|\alpha|^2 = |a + b\omega_K|^2/c$, and
\item $\gcd(|\alpha|^2, |\beta|^2, a, b) = 1$.
\end{enumerate}
\end{lemma}
\begin{proof}
By the definition of $a, b, c$, the statements (a) and (b) in the lemma are trivially true.
Also, we have
\[
\frac{|a + b\omega_K|^2}c
=
\frac{\alpha \beta^{\sigma} \cdot \alpha^{\sigma} \beta}{\beta\beta^{\sigma}}
=
\alpha\alpha^{\sigma},
\]
which proves (c) and (d).
It remains to prove (e).

Suppose that there exists a (rational) prime $p$ dividing $|\alpha|^2$ and $|\beta|^2$ in $\mathbb{Z}$.
Then $\alpha$ must be an element of a prime ideal, say, $\mathfrak{p}_{\alpha}$ of $\mathcal{O}_K$ lying above $p$.
Similarly, $\beta$ is an element of, say, $\mathfrak{p}_{\beta}$ lying above $p$.
Since $\alpha$ and $\beta$ are assumed to generate $\mathcal{O}_K$ it follows that
$\alpha\not\in\mathfrak{p}_{\beta}$
and
$\beta\not\in\mathfrak{p}_{\alpha}$.
In particular, $\mathfrak{p}_{\alpha} \neq \mathfrak{p}_{\beta}$. 
Since they are both prime ideals of $\mathcal{O}_K$ above $p$, 
we must have $\mathfrak{p}_{\alpha} = \mathfrak{p}_{\beta}^{\sigma}$
(and the prime $p$ must split in $\mathcal{O}_K$).
Consequently, $\beta^{\sigma} \not\in \mathfrak{p}_{\alpha}^{\sigma} = \mathfrak{p}_{\beta}$.

Now, suppose that $p$ additionally divides $a$ and $b$. 
Then 
$a + b\omega_K$ 
is in $p\mathcal{O}_K = \mathfrak{p}_{\alpha}\mathfrak{p}_{\beta} \subset \mathfrak{p}_{\beta}$.
However, we already proved that $\alpha\not\in\mathfrak{p}_{\beta}$ and $\beta^{\sigma}\not\in\mathfrak{p}_{\beta}$, so that $a + b\omega_K = \alpha\beta^{\sigma}\not\in\mathfrak{p}_{\beta}$.
This is a contradiction.
So, $p$ cannot divide both $a$ and $b$ and we obtain $\gcd(|\alpha|^2, |\beta|^2, a, b) = 1$.
\end{proof}

\subsection{The case $\mathbf{S} = \mathbf{S}^2_{\I}$}\label{sec:CaseS2_I}
The ambient space in this case is $W = \mathbb{R}^3$ and we define
\[
\begin{cases}
\mathbf{S} = \{ (x_1, x_2, x_3) \in \mathbb{R}^3 \mid x_1^2 + x_2^2 + x_3^2 = 1 \},\\
\mathbf{n} = (0, 0, 1), \\ 
\mathbf{P} = \{ (x_1, x_2, 0) \in \mathbb{R}^3 \mid x_1, x_2\in \mathbb{R} \}.
\end{cases}
\]
This gives
\begin{equation}
R = D = 1 %\quad \mathbf{o} = (0, 0, 0)
\end{equation}
and
\begin{equation}
\label{eq:Psi_case_A}
\Psi(x_1, x_2, x_3) =
\frac{
(2x_1, 2x_2, x_1^2 + x_2^2 +x_3^2- 1)}{
x_1^2 + x_2^2 + x_3^2 - 2x_3 + 1}.
\end{equation}

We define $\varphi : \mathbb{C} \longrightarrow \mathbf{P}$ to be
\begin{equation}
\label{eq:def_phi_A}
\varphi(u + \mathrm{i} v) = (u - v, u + v -1, 0).
\end{equation}

To prove ($\Phi$-i) in this case, we let 
$z_1 = u_1 + v_1\mathrm{i} 
$
and
$z_2 = u_2 + v_2\mathrm{i}$.
Then
\begin{align*}
| 
\varphi(z_1) 
-
\varphi(z_2) 
|^2
&=
| 
((u_1 -u_2) - (v_1 - v_2), (u_1 - u_2) + (v_1 - v_2), 0) 
|^2 \\
&=
2|z_1 - z_2|^2.
\end{align*}
So this proves ($\Phi$-i) with $C = \sqrt 2$.

For ($\Phi$-ii), let $r \in K$ and write $ r = \alpha/\beta = (a + b\mathrm{i})/c$ using the notations in Lemma~\ref{lem:reduced_fraction_in_Ok}.
To find $\HH_{\mathbf S \cap \mathbb{Q}^3}(\Phi(r))$, we compute $\Phi(r)$ using Lemma~\ref{lem:reduced_fraction_in_Ok}:
\begin{align*} 
\Phi(r) &= 
\Psi\left( \varphi \left( \frac{a + b\mathrm{i}}c\right)\right)
= 
\Psi\left(
\frac ac - \frac bc, \frac ac + \frac bc -1, 0
\right) \\
&=
\frac{
((a - b)c, (a + b - c)c, a^2 + b^2 - (a+b)c)
}{
a^2 + b^2 + c^2 - (a + b)c
} \\
&=
\frac{
(a - b, a + b - |\beta|^2, |\alpha|^2 - a-b)
}{
|\alpha|^2 + |\beta|^2 - a - b
}.  
\end{align*} 
And
\begin{multline}\label{eq:GCD1_case_1_I}
\gcd( a - b, a + b - |\beta|^2, |\alpha|^2 - a-b,|\alpha|^2  + |\beta|^2 - a - b) \\
= \gcd( |\alpha|^2, |\beta|^2, a - b, a + b).
\end{multline}
Note that the prime 2 does not split in $\mathbb{Z}[\mathrm{i}]$, so 2 cannot divide $\gcd(|\alpha|^2, |\beta|^2)$ (see the proof of Lemma~\ref{lem:reduced_fraction_in_Ok}).
Therefore we see that the gcd condition \eqref{eq:gcd_condition_p_and_q} for $\Phi(r)$ is satisfied by 
Lemma~\ref{lem:reduced_fraction_in_Ok} because
\begin{equation}\label{eq:GCD2_case_1_I}
\gcd( |\alpha|^2, |\beta|^2, a - b, a + b)  = \gcd( |\alpha|^2, |\beta|^2, a , b) = 1.
\end{equation}
Hence we obtain
\begin{equation*}
%\label{eq:Ht_case_A}
\HH_{\mathbf S\cap \mathbb{Q}^3}(\Phi(r))
=
|\alpha|^2 + |\beta|^2 - a - b.
\end{equation*}
On the other hand,
\begin{align*}
\frac{| \varphi(r) - \mathbf{n}|^2}{2RD} &=
\frac{| ( \frac ac - \frac bc, \frac ac + \frac bc -1, -1) |^2}{2} 
=\frac{a^2 + b^2 + c^2 - c(a + b)}{c^2} \\
&=
\frac{
|\alpha|^2  + |\beta|^2 - a - b
}{
c} \\
&=
\frac{
\HH_{\mathbf S\cap \mathbb{Q}^3}(\Phi(r))
}{
\HH_{K}(r)},
\end{align*}
which gives 
($\Phi$-ii) 
in this case.

\subsection{The case $\mathbf{S} = \mathbf{S}^2_{\II}$}\label{sec:CaseS2_II}
As in the case of $\mathbf{S}^2_{\I}$, the ambient space is again $W = \mathbb{R}^3$.
We define
\[
\begin{cases}
\mathbf{S} = \{ (x_1, x_2, x_3) \in \mathbb{R}^3 \mid x_1^2 + x_2^2 + x_3^2 = 2 \},\\
\mathbf{n} = (0, 1, 1), \\ 
\mathbf{P} = \{ (x_1, x_2, x_3) \in \mathbb{R}^3 \mid x_2 + x_3 = 1 \},
\end{cases}
\]
which gives
\begin{equation}
R = \sqrt2, \quad D = 1/\sqrt2 
%\quad \text{and} \quad \mathbf{o} = (0, \tfrac12, \tfrac12).
\end{equation}
and 
\begin{equation}
\label{eq:Psi_case_B}
\Psi(x_1, x_2, x_3) =
\frac{
(2x_1, x_1^2 + x_2^2 + x_3^2 - 2x_3, x_1^2 + x_2^2 + x_3^2 - 2x_2)}{
x_1^2 + x_2^2 + x_3^2 - 2x_2 - 2x_3 + 2}.
\end{equation}
In this case, we define $\varphi : \mathbb{C} \longrightarrow \mathbf{P}$ to be
\begin{equation}
\label{eq:def_phi_B}
\varphi(u + v\sqrt{-2}) = (1 - 2v, u, 1-u).
\end{equation}
%To prove ($\Phi$-i) in this case, we let 
Writing
$z_1 = u_1 + \sqrt{-2}v_1
$
and
$z_2 = u_2 + \sqrt{-2}v_2$,
we get
\begin{align*}
| 
\varphi(z_1) 
-
\varphi(z_2) 
|^2
&=
| 
(-2(v_1 - v_2), u_1 - u_2, -(u_1-u_2)) | ^2 \\
&=
2((u_1-u_2)^2 + 2(v_1 - v_2)^2) =
2|z_1 - z_2|^2,
\end{align*}
which proves ($\Phi$-i) with $C = \sqrt 2$.

For ($\Phi$-ii), we proceed similarly using Lemma~\ref{lem:reduced_fraction_in_Ok}.
Let $r = \alpha/\beta = (a + b\sqrt{-2})/c \in K = \mathbb{Q}(\sqrt{-2})$.
Then
\begin{align*}
\Phi(r) &= 
\Psi\left( \varphi \left( \frac{a + b\sqrt{-2}}c\right)\right)
=
\Psi\left(
1 - \frac{2b}c, \frac ac, 1 - \frac ac
\right) \\
&=
\frac{
((c - 2b)c, a^2 + 2b^2 - 2bc, a^2 + 2b^2 + c^2 - 2(a+b)c)
}{
a^2 + 2b^2 + c^2 - (a + 2b)c
} \\
&=
\frac{
(|\beta|^2 - 2b, |\alpha|^2 - 2b, |\alpha|^2 + |\beta|^2 - 2(a+b))
}{
|\alpha|^2 + |\beta|^2 - (a + 2b)
}.
\end{align*}
To check the gcd condition \eqref{eq:gcd_condition_p_and_q} for $\Phi(r)$, we note
\begin{multline}\label{eq:GCD1_case_1_II}
\gcd( 
|\beta|^2 - 2b, |\alpha|^2 - 2b, |\alpha|^2 + |\beta|^2 - 2(a+b),
|\alpha|^2 + |\beta|^2 - (a + 2b)
) 
\\
= \gcd( |\alpha|^2, |\beta|^2, a, 2b).
\end{multline}
Again, the prime 2 does not split in $\mathbb{Z}[\sqrt{-2}]$, so 2 cannot divide $\gcd(|\alpha|^2, |\beta|^2)$.
So we deduce from Lemma~\ref{lem:reduced_fraction_in_Ok} that
\begin{equation}\label{eq:GCD2_case_1_II}
\gcd( |\alpha|^2, |\beta|^2, a, 2b)  = \gcd( |\alpha|^2, |\beta|^2, a , b) = 1.
\end{equation}
So we obtain
\begin{equation*}
%\label{eq:Ht_case_A}
\HH_{\mathbf S\cap \mathbb{Q}^3}(\Phi(r))
=
|\alpha|^2  + |\beta|^2 - (a + 2b).
\end{equation*}
Also,
\begin{align*}
\frac{| \varphi(r) - \mathbf{n}|^2}{2RD} &=
\frac{| 
(1 - \tfrac{2b}c, \tfrac ac - 1, - \tfrac ac)
|^2}{2}
=\frac{a^2 + 2b^2 + c^2 - c(a + 2b)}{c^2}\\
&=
\frac{
|\alpha|^2  + |\beta|^2 - (a + 2b)
}{
c} \\
&=
\frac{
\HH_{\mathbf S\cap \mathbb{Q}^3}(\Phi(r))
}{
\HH_{K}(r)},
\end{align*}
which gives 
($\Phi$-ii) 
in this case.

\subsection{The case $\mathbf{S} = \mathbf{S}^2_{\III}$}\label{sec:CaseS2_III}
In this case, the ambient space $W$ is defined to be
\[
W
= \{ (x_0, x_1, x_2, x_3) \in \mathbb{R}^4 \mid
x_0 + x_1 + x_2 + x_3 = 1
\}.
\]
Also, we define
\[
\begin{cases}
\mathbf{S} = \left\{ (x_0, x_1, x_2, x_3) \in W
\, | \, 
x_0^2 + x_1^2 + x_2^2 + x_3^2 = 1
\right\},\\
\mathbf{n} = (0, 0, 0, 1), \\ 
\mathbf{P} = \{ (x_0, x_1, x_2, 0) \in W
\mid x_0 + x_1 + x_2 = 1 \}.
\end{cases}
\]
This yields
\begin{equation}
R = \sqrt3/2, \quad D = 2/\sqrt3 %\quad \text{and} \quad \mathbf{o} = (\tfrac13, \tfrac13, \tfrac13, 0).
\end{equation}
and %When $(x_1, x_2, x_3, 0) \in \mathbf{P}$, 
\begin{equation}
\label{eq:Psi_case_C}
\Psi(x_0, x_1, x_2, x_3) =
\frac{
(2x_0, 2x_1, 2x_2, x_0^2 + x_1^2 + x_2^2 + x_3^2- 1)}{
x_0^2 + x_1^2 + x_2^2 + x_3^2 - 2x_3 + 1}.
\end{equation}

We define $\varphi : \mathbb{C} \longrightarrow \mathbf{P}$ to be
%\begin{equation*}
%\boxed{
%\varphi(u + v\omega_K) = (u, v, 1-u-v, 0).
%}
%\end{equation*}
%\todo{This is the old definition, which we will remove later.}
\begin{equation}
\label{eq:def_phi_C}
\varphi(u + v\omega_K) = (1-u, u-v, v, 0).
\end{equation}
Recall that in this case $K = \mathbb{Q}(\sqrt{-3})$ and $\omega_K = (-1 + \sqrt{-3})/2$.
It is a simple exercise to show that
\[
|\varphi(z_1) - \varphi(z_2)|^2 = 2|z_1 - z_2|^2,
\]
which proves ($\Phi$-i) with $C = \sqrt 2$.

For ($\Phi$-ii), we write
$r = \alpha/\beta = (a + b\omega_K)/c$
and
\begin{align*}
\Phi(r) &= 
\Psi\left( \varphi \left( \frac{a + b\omega_K}c\right)\right)
=
\Psi\left(
1 - \frac ac, \frac ac - \frac bc, \frac bc, 0
\right) \\
&=
\frac{
(
(c-a)c, (a-b)c, bc, a^2 - ab + b^2 - ac
)
}{
a^2 - ab + b^2 + c^2 - ac
} \\
&=
\frac{
(
|\beta|^2 - a, a-b, b, |\alpha|^2 - a
)
}{
|\alpha|^2 + |\beta|^2 - a
}.
\end{align*}
The gcd condition \eqref{eq:gcd_condition_p_and_q} for $\Phi(r)$ is again satisfied by Lemma~\ref{lem:reduced_fraction_in_Ok} because
\begin{equation*}
\gcd( 
|\beta|^2 - a, a - b, b, |\alpha|^2 - a,
|\alpha|^2 + |\beta|^2 - a
)\\ 
= \gcd( |\alpha|^2, |\beta|^2, a, b)
=1.
\end{equation*}
So,
\begin{equation}
\label{eq:Ht_case_A_III}
\HH_{\mathbf S\cap \mathbb{Q}^4}(\Phi(r))
=
|\alpha|^2  + |\beta|^2 - a.
\end{equation}
And we have
\begin{align*}
\frac{| \varphi(r) - \mathbf{n}|^2}{2RD} &=
\frac{a^2 -ab + b^2 + c^2 - ac}{c^2} \\
&=
\frac{
|\alpha|^2  + |\beta|^2 - a 
}{
c} \\
&=
\frac{
\HH_{\mathbf S\cap \mathbb{Q}^4}(\Phi(r))
}{
\HH_{K}(r)}.
\end{align*}
This proves
($\Phi$-ii) 
in this case.

\section{Proof of Theorem~\ref{thm:horocycle_new}}\label{sec:proof_of_extension}
We continue to use the notations introduced in the beginning of \S\ref{Sec:Property_of_maps}.
Recall from \eqref{def:field_F} that $\mathbb{F}$ is by definition either $\mathbb{R}$ or $\mathbb{C}$,
depending on $n = 1$ or $n = 2$.
As in the introduction, we define $\mathbb{H}^{n+1}$ to be 
\[
\mathbb{H}^{n+1} = \{ (z,s) \in \mathbb F \times \mathbb R \mid z\in \mathbb{F}, s > 0 \},
\]
so that $\mathbb{H}^{n+1}$ is either the hyperbolic surface (when $n = 1$) or the hyperbolic 3-space (when $n = 2$).
In particular, we will identify $\mathbb{F}$ with the set $\mathbb{F} \times \{ 0 \}$.
Then $\hat{\mathbb{F}} = \mathbb{F} \cup \{\infty\}$  is the boundary of $\mathbb{H}^{n+1}$ and therefore 
$\overline{\mathbb{H}^{n+1}}  = \mathbb{H}^{n+1} \cup \hat{\mathbb{F}}$.
%Also, 

Now we define the maps $\overline{\varphi}$ and $\bar{\Psi}$, extending $\varphi$ and $\Psi$
(see Figure~\ref{fig:maps_between_spaces}).
\begin{figure}
\centering
\begin{tikzcd}
\overline{\mathbb{H}^{n+1}} \arrow{r}{\bar\varphi} & 
\overline{\mathbf{H}} \arrow{r}{\bar\Psi} & \overline{\mathbf{B}} \\
\mathbb{F} \arrow[hook]{u}{} \arrow{r}{\varphi} & 
\mathbf{P} 
\arrow[hook]{u}{} 
\arrow{r}{\Psi}
& \mathbf{S} \setminus \{ \mathbf{n} \}
\arrow{l}{}
\arrow[hook]{u}{} 
\end{tikzcd}
\caption{Maps between various spaces}
\label{fig:maps_between_spaces}
\end{figure}
Recall that $\varphi: \mathbb{F} \longrightarrow \mathbf{P}$ is a map (cf.~\eqref{def:map_varphi}) 
satisfying the conditions ($\Phi$-i) and ($\Phi$-ii) in Lemma~\ref{lem:main_lemma}.
%And $\Psi$ is the reflection in $\mathbf{P}$ of $W$. 
For $(z, s) \in \mathbb{H}^{n+1}$, we define $\bar{\varphi}: \overline{\mathbb{H}^{n+1}} \longrightarrow W$
to be
\[
\bar{\varphi}(z, s) = \varphi(z) + \frac{Cs}{R}(\mathbf{n} - \mathbf{c}).
\]
Here, $C$ is the same constant appearing in ($\Phi$-i) of Lemma~\ref{lem:main_lemma}.

To define $\bar{\Psi}$, we let $\mathbf H$ be the connected component of $W \setminus \mathbf P$ containing $\mathbf n$.
Denote by 
$\mathrm{Ref}_{\mathbf P}$ 
and
$\mathrm{Ref}_{\mathbf S}$ 
the reflections of $\hat{W}$ in $\mathbf{P}$ and in $\tilde{\mathbf{S}}$, respectively.
Then we define
\begin{equation}
\bar \Psi = \mathrm{Ref}_{\tilde{\mathbf S}} \circ \mathrm{Ref}_{\mathbf P}.
\end{equation}
Then it is clear that $\bar{\Psi}$ maps $\bar{\mathbf H}$ onto $\bar{\mathbf B}$ and
it is a hyperbolic isometry between them.
%under the standard hyperbolic model of the ball and the upper half space.
Also, it is easy to see that $\overline{\varphi}$ and $\bar{\Psi}$ extend $\varphi$ and $\Psi$, as indicated in Figure~\ref{fig:maps_between_spaces}.

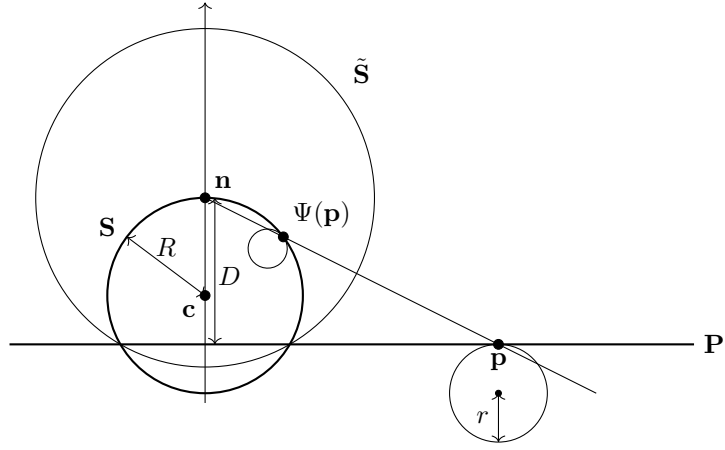
\begin{figure}
\begin{tikzpicture}[scale=1.3]
\draw[thick] (-2, 0) -- (5, 0) node[right] {$\mathbf P$};
\draw[->] (0, -.6) -- (0, 3.5);
% R = 1, d = 1/2
\draw (0,1.5) circle (1.732); 
%node [above left] {$\mathbf n$}; 
\draw[thick] (0,.5) circle (1); %node [Right] {$\mathbf S$}; 

\draw (0,1.5) -- (4,-.5);

\node at (-1,1.2) {$\mathbf S$};
\node at (1.6,2.8) {$\tilde{\mathbf S}$};
\node at (-.4,1) {$R$};
\node at (.25,.7) {$D$};

\draw[<->] (0,.5) -- (-.8,1.1);
\draw[<->] (.1,1.5) -- (.1,0);

\draw[fill] (0,1.5) circle (0.05) node [above right] {$\mathbf n$};
\draw[fill] (0, .5) circle (0.05) node [below left] {$\mathbf c$};
\draw[fill] (3,0) circle (0.05) node [below] {$\mathbf p$};
%	\draw[fill] (0,0) circle (0.05) node [above left] {$\mathbf o$};

\draw[fill] (.8,1.1) circle (0.05) node [above right] {$\Psi(\mathbf p)$};

\draw (.8,1.1) arc (36.87:36.87+360:0.2); 
%\draw (3,.5) circle (0.5);
\draw[fill] (3,-0.5) circle (0.03);
\draw (3,-.5) circle (0.5); 
\draw[<->] (3, -0.5) -- (3, -1) node[midway, left]{$r$};

\end{tikzpicture}
\caption{Changes of radius under the reflection $\Psi$ in $\tilde{\mathbf{S}}$.}\label{projection2}
\end{figure}
%We define a map
%$\Psi:\hat{W}^3 \longrightarrow  \hat{W}^3$
%by composing them as
%\[
%\Psi = \mathrm{Ref}_{\tilde{\mathbf S}} \circ \mathrm{Ref}_{\mathbf P}.
%\]

\begin{lemma}\label{lem_horo}
The notations are the same as in \S\ref{Sec:Property_of_maps}.
Additionally, 
let $r$ be the radius of a reflected horosphere by $\mathrm{Ref}_{\mathbf P}$ based at $\mathbf{p}$ tangent to $\mathbf{P}$,
which is on the opposite side to $\mathbf{n}$ with respect to $\mathbf{P}$.
(See Figure~\ref{projection2}.)
Then the radius $\rho$ of the horosphere based at $\Psi(\mathbf{p})$ tangent to $\mathbf{S}$ is given by the formula
\[
\rho = \frac{2 r RD}{| \mathbf p - \mathbf n |^2 +2 r D}.
\]
\end{lemma}
\begin{proof}
The distance between $\mathbf{n}$ and the center of the horosphere (in the opposite side of $\mathbf{n}$) based at $\mathbf p\in \mathbf{P}$ with radius $r$ is
\[
\sqrt{ (D +r)^2 - D^2 + |\mathbf p - \mathbf n|^2}.
\]
Thus the distances from $\mathbf{n}$ to the closest and farthest points on the horosphere are
\[
\sqrt{ (D+ r)^2 - D^2 + |\mathbf p - \mathbf n|^2} \pm r.
\]
Recall that the radius of $\tilde{\mathbf{S}}$ is $\sqrt{2RD}$.
From a general property of the reflection $\Psi$, we see that the diameter of the image of the horosphere under $\Psi$ has diameter
\begin{multline*}
\frac{2RD}{\sqrt{ (D + r)^2 - D^2 + | \mathbf p - \mathbf n |^2} - r} - \frac{2RD}{\sqrt{ (D + r)^2 - D^2 + | \mathbf p - \mathbf n |^2} + r} \\
= \frac{4rRD}{(D + r)^2 - D^2 + |\mathbf p - \mathbf n|^2 - r^2}.
\end{multline*}
Simplifying this expression, we obtain the formula for $\rho$, which completes the proof.
\end{proof}

Using Lemma~\ref{lem_horo},
it is now straightforward to prove the formula \eqref{eq:rho} for $\rho$ in Theorem~\ref{thm:horocycle_new}.
Suppose that there is a Ford horosphere in $\mathbb{H}^{n+1}$ based at $z\in K$ with radius $1/(2\HH_{K}(z))$.
From ($\Phi$-i), 
we see that this horosphere is mapped by the extended map $\overline{\varphi}$
to a horosphere with radius $C/(2\HH_{K}(z))$ based at $\varphi(z)$.
Then, after we further apply $\mathrm{Ref}_{\mathbf P}$, we are in the situation described in Lemma~\ref{lem_horo} with $\mathbf{p} = \varphi(z)$ and $r = C/(2\HH_{K}(z))$. 
We use Lemma~\ref{lem_horo} with this and simplify using the condition ($\Phi$-ii) to obtain \eqref{eq:rho}.

It remains to show the last assertion in Theorem~\ref{thm:horocycle_new},
namely, two horospheres on $\mathbb F$ %$\mathbf{S}$
based at $z, z' \in K$
%$\mathbf{z} = \Phi(z)$ and $\mathbf{z}' = \Phi(z')$ 
are either disjoint or tangent.
Let $z, z'$ be two distinct points of $K$ 
and write $z=\alpha/\beta$ and $z'= \alpha'/\beta'$ as before.
%\sout{with $\alpha, \beta$ and $\alpha', \beta'$ generating $\mathcal{O}_K$.}
Since the horospheres based at $z$, $z'$ have radius $1/(2|\beta|^2)$, $1/(2|\beta'|^2)$,
the square of the distance between the centers of the horospheres based at $z$ and $z'$ is
\begin{align*}
\left( \frac{1}{2|\beta|^2} - \frac{1}{2|\beta'|^2} \right)^2 + \left| \frac{\alpha}{\beta} - \frac{\alpha'}{\beta'} \right|^2 
&= \left( \frac{1}{2|\beta|^2} + \frac{1}{2|\beta'|^2} \right)^2 + \frac{| \alpha \beta' - \alpha' \beta |^2 -1}{|\beta|^2 |\beta'|^2} \\
&\ge \left( \frac{1}{2|\beta|^2} + \frac{1}{2|\beta'|^2} \right)^2.
\end{align*}
Here, we use the fact that $| \alpha \beta' - \alpha' \beta |^2$ is at least 1. 
For the case of $n = 1$, see Figures \ref{fig:S1I}, \ref{fig:S1II} and \ref{fig:S1III}.
This completes the proof of Theorem~\ref{thm:horocycle_new}.

The horosphere based at $\Phi(z)$ with (Euclidean) radius $\rho = R/(1+2(R/C)\HH_{\mathcal{Z}}(\mathbf z) )$ is 
centered at 
\[
\mathbf c + \left( 1 - \frac{\rho}{R}\right) (\mathbf z - \mathbf c) 
= \mathbf c + \frac{ 2R \HH_{\mathcal{Z}}(\mathbf z)(\mathbf z - \mathbf c)}{C + 2R \HH_{\mathcal{Z}}(\mathbf z)},
\]
where $\mathbf c$ is the center of the sphere $\mathbf S$.
Hence, the square of the distance between the centers of the horospheres based at $\mathbf z$ and $\mathbf z'$ is 
\begin{multline*}
\left| \frac {2 R h (\mathbf z - \mathbf c)}{C + 2 R h} - \frac {2 R h (\mathbf z' - \mathbf c)}{C + 2 R h'} \right|^2 \\
= \left(\frac {2 R^2 h}{C + 2 R h} \right)^2 + \left(\frac {2R^2 h'}{C + 2 R h'}\right)^2 - \frac {8 R^2 h h' (\mathbf z - \mathbf c)\cdot (\mathbf z' - \mathbf c)}{(C + 2 R h)(C + 2 R h')}.
\end{multline*}
Here, we used 
$h = \HH_{\mathcal{Z}}(\mathbf{z})$
and
$h' = \HH_{\mathcal{Z}}(\mathbf{z}')$
to lighten the notations.
On the other hand, the square of the sum of radii of the horospheres based at $\mathbf z$ and $\mathbf z'$ is
\begin{align*}
\left(\rho_{\mathbf z} + \rho_{\mathbf z'} \right)^2 
&=\left(\frac{RC}{C + 2 R h} + \frac{RC}{C + 2 R h'}\right)^2 \\
%&=\left(\frac {2 R^2 h}{C + 2 R h} - \frac {2 R^2 h'}{C + 2 R h'}\right)^2 + \frac {4R^2C^2 }{(C + 2 R h)(C + 2 R h')} \\
&= \left(\frac {2 R^2 h}{C + 2 R h} \right)^2 + \left(\frac {2 R^2 h'}{C + 2 R h'}\right)^2 - \frac {8 R^2 h h'(R^2 - C^2/(2hh') ) }{(C + 2 R h)(C + 2 R h')}.
\end{align*}
Therefore, we conclude that 
\begin{equation}\label{ineq:claim}
(\mathbf z - \mathbf c)\cdot (\mathbf z' - \mathbf c) \le R^2 - \frac{C^2}{2\HH_{\mathbf S}(\mathbf z)\HH_{\mathbf S}(\mathbf z')}.
\end{equation}
Moreover, two horospheres based at $\mathbf z$ and $\mathbf z'$ are tangent if and only if the equality in \eqref{ineq:claim} holds true.

%\todo[inline]{Mention something along the line of ``for the rest of the paper, we will show how horocycles change and \eqref{ineq:claim} applies''}

For the rest of the paper, we give an explicit description of \eqref{ineq:claim} in terms of the (Euclidean) coordinates in their ambient spaces for each of the six spaces we consider. 

\subsection{The case $\mathbf{S} = \mathbf{S}_\I^1$}

\begin{figure}
\begin{tikzpicture}[scale=4]

\draw[thick] (0,0) circle (1);

\draw[ultra thick] (1,0) node [right] {$(\frac11,\frac01)$} arc (0:0+360:{1/(1+sqrt(2))}); 
\draw[ultra thick] (0,1) node [above] {$(\frac01,\frac11)$} arc (-270:-270+360:{1/(1+sqrt(2))}); 
\draw[ultra thick] (-1,0) node [left] {$(\frac{-1}1,\frac01)$} arc (-180:-180+360:{1/(1+sqrt(2))}); 
\draw[ultra thick] (0,-1) node [below] {$(\frac01,\frac{-1}1)$} arc (-90:-90+360:{1/(1+sqrt(2))}); 

\draw (1,0) arc (0-90:-270+90:1);
\draw (0,1) arc (-270-90+360:-180+90:1);
\draw (-1,0) arc (-180-90+360:-90+90:1);
\draw (0,-1) arc (-90-90+360:0+90:1);

\draw[ultra thick] (4/5,3/5) node [right=2pt,yshift=-1pt] {$(\frac 45, \frac 35)$} arc (36.87:36.87+360:{1/(1+5*sqrt(2))}); 
\draw[ultra thick] (3/5,4/5) node [above=4pt,xshift=3pt] {$(\frac 35, \frac 45)$} arc (53.13:53.13+360:{1/(1+5*sqrt(2))}); 

\draw (1,0) arc (0-90:36.87+90-360:1/3);
\draw (4/5,3/5) arc (36.87-90:53.13+90-360:1/7);
\draw (3/5,4/5) arc (53.13-90:-270+90:1/3);

\draw[ultra thick] (12/13,5/13) node [right=3pt, yshift=-3pt] {$(\frac{12}{13},\frac 5{13})$
} arc (22.62:22.62+360:{1/(1+13*sqrt(2))}); 
\draw[ultra thick] (15/17,8/17) node [right=2pt] {$(\frac{15}{17},\frac{8}{17})$
} arc (28.07:28.07+360:{1/(1+17*sqrt(2))}); 

\draw[ultra thick] (4/5,-3/5) node [right] {$(\frac 45, \frac{-3}5)$} arc (-36.87:-36.87+360:{1/(1+5*sqrt(2))}); 
\draw[ultra thick] (3/5,-4/5) node [right,yshift=-5pt] {$(\frac 35, \frac{-4}5)$} arc (-53.13:-53.13+360:{1/(1+5*sqrt(2))}); 

\draw (1,0) arc (0+90:-36.87-90+360:1/3);
\draw (4/5,-3/5) arc (-36.87+90:-53.13-90+360:1/7);
\draw (3/5,-4/5) arc (-53.13+90:-90-90+360:1/3);

\draw (1,0) arc (0-90:22.62+90-360:1/5);
\draw (12/13,5/13) arc (22.62-90:28.07+90-360:1/21);
\draw (15/17,8/17) arc (28.07-90:36.87+90-360:1/13);

\draw[ultra thick] (21/29,20/29) node [right,yshift=2pt] {$(\frac{21}{29},\frac{20}{29})$} arc (43.60:43.60+360:{1/(1+29*sqrt(2))});
\draw[ultra thick] (20/29,21/29) node [above, xshift=12pt] {$(\frac{20}{29},\frac{21}{29})$
} arc (46.40:46.40+360:{1/(1+29*sqrt(2))}); 

\draw (4/5,3/5) arc (36.87-90:43.60+90-360:1/17);
\draw (21/29,20/29) arc (43.60-90:46.40+90-360:1/41);
\draw (20/29,21/29) arc (46.40-90:53.13+90-360:1/17);

\end{tikzpicture}

\begin{tikzpicture}[scale=2]
\draw[thick] (-1.7, 0) -- (3.7, 0);

\draw[ultra thick] (-1.7,1.4142) -- (3.7,1.4142); 
\draw[ultra thick] (-1,0) node [below] {$0 = \Phi^{-1} ( \frac {-1}1,\frac 01)$} arc (-90:270:0.7071); 
\draw[ultra thick] (1,0) node [below] {$\sqrt 2 = \Phi^{-1} (\frac 11,\frac 01)$} arc (-90:270:0.70711); 
\draw[ultra thick] (3,0) node [below] {$2\sqrt 2 = \Phi^{-1} (\frac 35, \frac 45)$} arc (-90:270:0.7071); 

\draw (-1,1.6) -- (-1,0);
\draw (1,1.6) -- (1,0);
\draw (3,1.6) -- (3,0);

\draw[ultra thick] (0,0) node [below=15pt] {$\frac{1}{\sqrt 2} = \Phi^{-1}(\frac 01, \frac {-1}1)$} arc (-90:270:0.35355); 
\draw[ultra thick] (2,0) node [below=15pt] {$\frac{3}{\sqrt 2} = \Phi^{-1}(\frac 45, \frac 35)$} arc (-90:270:0.35355); 

\draw (-1,0) arc (180:0:1/2);
\draw (0,0) arc (180:0:1/2);
\draw (1,0) arc (180:0:1/2);
\draw (2,0) arc (180:0:1/2);

\draw[ultra thick] (1/2,0) %node [below] {$\frac{1}{2}$} 
arc (-90:270:0.08839); 
\draw[ultra thick] (1/3,0) %node [below] {$\frac{1}{3}$} 
arc (-90:270:0.07857); 

\draw[ultra thick] (3/2,0) %node [below] {$\frac{3}{2}$} 
arc (-90:270:0.08839); 
\draw[ultra thick] (5/3,0) %node [below] {$\frac{5}{3}$} 
arc (-90:270:0.07857); 

\draw (0,0) arc (180:0:1/6);
\draw (1/3,0) arc (180:0:1/12);
\draw (1/2,0) arc (180:0:1/4);

\draw (1,0) arc (180:0:1/4);
\draw (3/2,0) arc (180:0:1/12);
\draw (5/3,0) arc (180:0:1/6);

\draw[ultra thick] (7/3,0) %node [below] {$\frac{7}{3}$} 
arc (-90:270:0.07857); 
\draw[ultra thick] (5/2,0) %node [below] {$\frac{5}{2}$} 
arc (-90:270:0.08839); 

\draw (2,0) arc (180:0:1/6);
\draw (7/3,0) arc (180:0:1/12);
\draw (5/2,0) arc (180:0:1/4);
\end{tikzpicture}

\caption{Horocycles for the circle $\mathbf S^1_\I$. The radius of the horocycle at $(\frac ac, \frac bc)$ is $1/(1+\sqrt 2 c)$ on $\mathbf S^1_\I$ (above). 
The projected horocycle at $\Phi^{-1}(\frac ac, \frac bc)$ on $\mathbb R$ has radius $1/(2(c-b))$ (below).}\label{fig:S1I}
\end{figure}

A rational point $\mathbf z = (\frac ac, \frac bc)$ on $\mathbf S^1_\I$ has height $\HH(\mathbf z) = c$, the Ford horocycle based at $\mathbf z$ has radius $1/(1+\sqrt 2 c)$.
For $\mathbf z = (\frac ac, \frac bc)$ and $\mathbf z' = (\frac {a'}{c'}, \frac {b'}{c'})$,
the inequality \eqref{ineq:claim} implies that 
\begin{align*}
(\mathbf z - \mathbf c)\cdot (\mathbf z' - \mathbf c)
&= \frac{aa' + bb'}{cc'} 
\le 1 - \frac{1}{cc'}.
\end{align*}
Therefore, we have
$aa' + bb' - cc' \le -1$
and two horocycles based at $\mathbf z = (\frac ac, \frac bc)$ and $\mathbf z' = (\frac {a'}{c'}, \frac {b'}{c'})$ are tangent if and only if 
$aa' + bb' - cc' = -1$.
%See Figure~\ref{fig:S1I} for some horocycles on $\mathbf B$ and projected horocycles on $\mathbb H^2$.
In Figure~\ref{fig:S1I}, we give pictures of horospheres based at the following rational points on $\mathbf S^1_\I$ and corresponding points on $\mathbb R$:
\begin{align*}
 &  \left( \tfrac {-1}1,\tfrac 01\right)=\Phi\left(0\right), &
 & \left(\tfrac 01, \tfrac {-1}1\right)=\Phi\left(\tfrac{1}{\sqrt 2}\right), &
 & \left(\tfrac 35, \tfrac {-4}5\right)=\Phi\left(\tfrac{2\sqrt2}{3}\right), \\
 & \left(\tfrac 45, \tfrac {-3}5\right)=\Phi\left(\tfrac{3}{2\sqrt 2}\right), &
 & \left(\tfrac 11,\tfrac 01\right)=\Phi\left(\sqrt 2\right), &
 & \left(\tfrac{12}{13}, \tfrac{5}{13}\right)=\Phi\left(\tfrac{5}{2\sqrt 2}\right), \\
 & \left(\tfrac{15}{17}, \tfrac{8}{17}\right)=\Phi\left(\tfrac{4\sqrt2}{3}\right), &
 & \left(\tfrac 45, \tfrac 35\right)=\Phi\left(\tfrac{3}{\sqrt 2}\right), &
 & \left(\tfrac{21}{29}, \tfrac{20}{29}\right)=\Phi\left(\tfrac{5\sqrt2}{3}\right), \\
 & \left(\tfrac{20}{29}, \tfrac{21}{29}\right)=\Phi\left(\tfrac{7}{2\sqrt 2}\right), &
 & \left(\tfrac 35, \tfrac 45\right)=\Phi\left(2\sqrt 2\right), & 
 & \left(\tfrac 01, \tfrac 11\right)=\Phi(\infty).
\end{align*}
Note that the height of $\Phi^{-1} \left( \frac a{c}, \frac b{c} \right)$ is $c-b$.

\subsection{The case $\mathbf{S} = \mathbf{S}_\II^1$}

\begin{figure}
\begin{tikzpicture}[scale=2.8284]

\draw[thick]
(0,0) circle ({sqrt(2)});

\draw[ultra thick] (1,1) node [above right] {$(\frac11,\frac11)$} arc (45:45+360:{2-sqrt(2)}); 
\draw[ultra thick] (-1,1) node [above left] {$(\frac{-1}1,\frac11)$} arc (135:135+360:{2-sqrt(2)}); 
\draw[ultra thick] (1,-1) node [below right] {$(\frac11,\frac{-1}1)$} arc (-45:-45+360:{2-sqrt(2)}); 
\draw[ultra thick] (-1,-1) node [below left] {$(\frac{-1}1,\frac{-1}1)$} arc (-135:-135+360:{2-sqrt(2)}); 

\draw (1,1) arc (45-90:135+90-360:{sqrt(2)});
\draw (-1,1) arc (135-90:-135+90:{sqrt(2)});
\draw (-1,-1) arc (-135-90:-45+90-360:{sqrt(2)});
\draw (1,-1) arc (-45-90:45+90-360:{sqrt(2)});

%\draw[ultra thick] (1/5,7/5) node [above] {$(\frac 15, \frac 75)$} arc (81.87:81.87+360:{(10-sqrt(2))/49}); 
%\draw[ultra thick] (-1/5,7/5) node [above] {$( -\frac 15, \frac 75)$} arc (98.13:98.13+360:{(10-sqrt(2))/49}); 

%\draw (1,1) arc (45-90:81.87+90-360:{sqrt(2)/3});
%\draw (1/5,7/5) arc (81.87-90:98.13+90-360:{sqrt(2)/7});
%\draw (-1/5,7/5) arc (98.13-90:135+90-360:{sqrt(2)/3});

\draw[ultra thick] (1/5,-7/5) node [below right] {$(\frac 15, \frac{-7}5)$} arc (-81.87:-81.87+360:{(10-sqrt(2))/49}); 
\draw[ultra thick] (-1/5,-7/5) node [below left] {$(\frac{-1}5, \frac{-7}5)$} arc (-98.13:-98.13+360:{(10-sqrt(2))/49}); 

\draw (-1,-1) arc (-135-90:-98.13+90-360:{sqrt(2)/3});
\draw (-1/5,-7/5) arc (-98.13-90:-81.87+90-360:{sqrt(2)/7});
\draw (1/5,-7/5) arc (-81.87-90:-45+90-360:{sqrt(2)/3});

\draw[ultra thick]
(7/5,1/5) node [right, yshift=8pt] {$(\frac 75, \frac 15)$} arc (8.13:8.13+360:{(10-sqrt(2))/49}); 
\draw[ultra thick]
(7/5,-1/5) node [right, yshift=-5pt] {$(\frac 75, \frac{-1}5)$} arc (-8.13:-8.13+360:{(10-sqrt(2))/49}); 

\draw (1,-1) arc (-45-90:-8.13+90-360:{sqrt(2)/3});
\draw (7/5,-1/5) arc (-8.13-90:8.13+90-360:{sqrt(2)/7});
\draw (7/5,1/5) arc (8.13-90:45+90-360:{sqrt(2)/3});

\draw[ultra thick]
(17/13,-7/13) node [below right] {$(\frac{17}{13},\frac{-7}{13})$} arc (-22.38:-22.38+360:{(26-sqrt(2))/337}); 
\draw[ultra thick]
(23/17,-7/17) node [right, yshift=-4pt] {$(\frac{23}{17},\frac{-7}{17})$} arc (-16.93:-16.93+360:{(34-sqrt(2))/577}); 

\draw (1,-1) arc (-45-90:-22.38+90-360:{sqrt(2)/5});
\draw (17/13,-7/13) arc (-22.38-90:-16.93+90-360:{sqrt(2)/21});
\draw (23/17,-7/17) arc (-16.93-90:-8.13+90-360:{sqrt(2)/13});

\draw[ultra thick]
(41/29,1/29) node [right, yshift=6pt] {$(\frac{41}{29},\frac1{29})$
} arc (1.397:1.397+360:{(58-sqrt(2))/1681});
\draw[ultra thick]
(41/29,-1/29) node [right, yshift=-4pt] {$(\frac{41}{29}, \frac{-1}{29})$
} arc (-1.397:-1.397+360:{(58-sqrt(2))/1681}); 

\draw (7/5,-1/5) arc (-8.13-90:-1.397+90-360:{sqrt(2)/17});
\draw (41/29,-1/29) arc (-1.397-90:1.397+90-360:{sqrt(2)/41});
\draw (41/29,1/29) arc (1.397-90:8.13+90-360:{sqrt(2)/17});

\end{tikzpicture}

\begin{tikzpicture}[scale=2]
%	\node[below left] at (0, 0) {$O$};

\draw[thick] (-1.7, 0) -- (3.7, 0);

\draw[ultra thick] (-1.7,1.4142) -- (3.7,1.4142); 
\draw[ultra thick] (-1,0) node [below] {$0 = \Phi^{-1}  (\frac{-1}{1},\frac 11)
$} arc (-90:270:0.7071); 
\draw[ultra thick] (1,0) node [below] {$\sqrt 2 = \Phi^{-1} (\frac 11,\frac{-1}1)
$} arc (-90:270:0.70711); 
\draw[ultra thick] (3,0) node [below] {$2\sqrt 2 = \Phi^{-1} (\frac 75,\frac 15)$} arc (-90:270:0.7071); 

\draw (-1,1.6) -- (-1,0);
\draw (1,1.6) -- (1,0);
\draw (3,1.6) -- (3,0);

\draw[ultra thick] (0,0) node [below=15pt] {$\frac 1{\sqrt 2} =\Phi^{-1} (\frac{-1}{1},\frac{-1}{1})
$} arc (-90:270:0.35355); 
\draw[ultra thick] (2,0) node [below=15pt] {$\frac 3{\sqrt 2} =\Phi^{-1} (\frac 75,\frac{-1}5)
$} arc (-90:270:0.35355); 

\draw (-1,0) arc (180:0:1/2);
\draw (0,0) arc (180:0:1/2);
\draw (1,0) arc (180:0:1/2);
\draw (2,0) arc (180:0:1/2);

\draw[ultra thick] (1/2,0) node [below] {%$(\frac{1}{2},-\frac 12)$
} arc (-90:270:0.08839); 
\draw[ultra thick] (1/3,0) node [below] {%$\frac{1}{3}$
} arc (-90:270:0.07857); 

\draw[ultra thick] (3/2,0) node [below] {%$(\frac{3}{2},-\frac 32)$
} arc (-90:270:0.08839); 
\draw[ultra thick] (5/3,0) node [below] {%$\frac{5}{3}$
} arc (-90:270:0.07857); 

\draw (0,0) arc (180:0:1/6);
\draw (1/3,0) arc (180:0:1/12);
\draw (1/2,0) arc (180:0:1/4);

\draw (1,0) arc (180:0:1/4);
\draw (3/2,0) arc (180:0:1/12);
\draw (5/3,0) arc (180:0:1/6);

\draw[ultra thick] (7/3,0) node [below] {%$\frac{7}{3}$
} arc (-90:270:0.07857); 
\draw[ultra thick] (5/2,0) node [below] {%$\frac{5}{2}$
} arc (-90:270:0.08839); 

\draw (2,0) arc (180:0:1/6);
\draw (7/3,0) arc (180:0:1/12);
\draw (5/2,0) arc (180:0:1/4);
\end{tikzpicture}

\caption{Horocycles for the circle $\mathbf S^1_\II$. 
The radius of the horocycle at $(\frac ac, \frac bc)$ is $\sqrt 2/(1+\sqrt 2 c)$ on $\mathbf S^1_\II$ (above).
The projected horocycle at $\Phi^{-1}(\frac ac, \frac bc)$ on $\mathbb R$  
has radius $1/(2c-a-b)$ (below).
}\label{fig:S1II}
\end{figure}

A rational point $\mathbf z = (\frac ac, \frac bc)$ on $\mathbf S^1_\II$ has height $\HH(\mathbf z) = c$, the Ford horocycle based at $\mathbf z$ has radius $\sqrt 2/(1+\sqrt 2 c)$.
For $\mathbf z = (\frac ac, \frac bc)$ and $\mathbf z' = (\frac {a'}{c'}, \frac {b'}{c'})$,
the inequality \eqref{ineq:claim} implies that 
\begin{align*}
(\mathbf z - \mathbf c)\cdot (\mathbf z' - \mathbf c)
&= \frac{aa' + bb'}{cc'} 
\le 2 - \frac{2}{cc'}.
\end{align*}
Therefore, we have
$aa' + bb' - 2cc' \le - 2$
and two horocycles based at $\mathbf z = (\frac ac, \frac bc)$ and $\mathbf z' = (\frac {a'}{c'}, \frac {b'}{c'})$ are tangent if and only if 
$aa' + bb' - cc' = -2$.
%See Figure~\ref{fig:S1II} for some horocycles on $\mathbf B$ and projected horocycles on $\mathbb H^2$.
In Figure~\ref{fig:S1II}, we give pictures of horospheres based at the following rational points on $\mathbf S^1_\II$ and corresponding points on $\mathbb R$:
\begin{align*}
 &  \left( \tfrac {-1}1,\tfrac 11\right)=\Phi\left(0\right), &
 & \left(\tfrac {-1}1, \tfrac {-1}1\right)=\Phi\left(\tfrac{1}{\sqrt 2}\right), &
 & \left(\tfrac {-1}5, \tfrac {-7}5\right)=\Phi\left(\tfrac{2\sqrt2}{3}\right), \\
 & \left(\tfrac 15, \tfrac {-7}5\right)=\Phi\left(\tfrac{3}{2\sqrt 2}\right), &
 & \left(\tfrac 11,\tfrac {-1}1\right)=\Phi\left(\sqrt 2\right), &
& \left(\tfrac{17}{13}, \tfrac{-7}{13}\right)=\Phi\left(\tfrac{5}{2\sqrt 2}\right) , \\
 & \left(\tfrac{23}{17}, \tfrac{-7}{17}\right)=\Phi\left(\tfrac{4\sqrt2}{3}\right), &
 & \left(\tfrac 75, \tfrac {-1}5\right)=\Phi\left(\tfrac{3}{\sqrt 2}\right), &
 & \left(\tfrac{41}{29}, \tfrac{-1}{29}\right)=\Phi\left(\tfrac{5\sqrt2}{3}\right), \\
 & \left(\tfrac{41}{29}, \tfrac{1}{29}\right)=\Phi\left(\tfrac{7}{2\sqrt 2}\right), &
 & \left(\tfrac 75, \tfrac 15\right)=\Phi\left(2\sqrt 2\right), &
 & \left(\tfrac 11, \tfrac 11\right)=\Phi(\infty).
\end{align*}
Note that the height of $\Phi^{-1} \left( \frac a{c}, \frac b{c} \right)$ is $c-(a+b)/2$.

\subsection{The case $\mathbf{S} = \mathbf{S}_\III^1$}

\begin{figure}
\begin{tikzpicture}[scale=4]

\draw[thick] (0,0) circle (1);

% x = - sqrt(3) a /2  + sqrt(3) b/2, 
% x = -a/2 -b/2 + c, 
% a + b + c = 1

\draw[ultra thick] (0,1) node [above] {$(\frac01,\frac01,\frac11)$} arc (90:90+360:{1/(1+2/sqrt(3)}); 
\draw[ultra thick] ({-sqrt(3)/2},-1/2) node [below left] {$(\frac11,\frac01,\frac01)$} arc (-150:-150+360:{1/(1+2/sqrt(3)}); 
\draw[ultra thick] ({sqrt(3)/2},-1/2) node [below right] {$(\frac01,\frac11,\frac01)$} arc (-30:-30+360:{1/(1+2/sqrt(3)}); 

\draw (0,1) arc (90-90:-150+90:{sqrt(3)});
\draw ({-sqrt(3)/2},-1/2) arc (-150-90:-30+90-360:{sqrt(3)});
\draw ({sqrt(3)/2},-1/2) arc (-30-90:90+90-360:{sqrt(3)});

\draw[ultra thick] (0,-1) node [below] {$(\frac 23, \frac 23, \frac{-1}3)$} arc (-90:-90+360:{1/(1+6/sqrt(3)}); 
\draw[ultra thick] ({sqrt(3)/2},1/2) node [above right] {$(\frac{-1}3, \frac 23, \frac 23)$} arc (30:30+360:{1/(1+6/sqrt(3)}); 
\draw[ultra thick] ({-sqrt(3)/2},1/2) node [above left] {$(\frac 23, \frac{-1}3, \frac 23)$} arc (150:150+360:{1/(1+6/sqrt(3)}); 

\draw (0,1) arc (90-90:150+90-360:{1/sqrt(3)});
\draw ({-sqrt(3)/2},1/2) arc (150-90:-150+90:{1/sqrt(3)});
\draw ({-sqrt(3)/2},-1/2) arc (-150-90:-90+90-360:{1/sqrt(3)});
\draw (0,-1) arc (-90-90:-30+90-360:{1/sqrt(3)});
\draw ({sqrt(3)/2},-1/2) arc (-30-90:30+90-360:{1/sqrt(3)});
\draw ({sqrt(3)/2},1/2) arc (30-90:90+90-360:{1/sqrt(3)});

\draw[ultra thick] ({-3*sqrt(3)/14},-13/14) node [below left] {$(\frac67,\frac37,\frac{-2}7)$} arc (-111.79:-111.79+360:{1/(1+14/sqrt(3)}); 
\draw[ultra thick] ({3*sqrt(3)/14},-13/14) node [below right] {$(\frac37,\frac67,\frac{-2}7)$} arc (-68.21:-68.21+360:{1/(1+14/sqrt(3)}); 
\draw[ultra thick] ({4*sqrt(3)/7},1/7) node [right] {$(\frac{-2}7,\frac67,\frac37)$} arc (8.21:8.21+360:{1/(1+14/sqrt(3)}); 
\draw[ultra thick] ({5*sqrt(3)/14},11/14) node [above right] {$(\frac{-2}7,\frac37,\frac67)$} arc (51.79:51.79+360:{1/(1+14/sqrt(3)}); 

\draw ({-sqrt(3)/2},-1/2) arc (-150-90:-111.79+90-360:{sqrt(3)/5});
\draw ({-3*sqrt(3)/14},-13/14) arc (-111.79-90:-90+90-360:{sqrt(3)/9});
\draw (0,-1) arc (-90-90:-68.21+90-360:{sqrt(3)/9});
\draw ({3*sqrt(3)/14},-13/14) arc (-68.21-90:-30+90-360:{sqrt(3)/5});

\draw ({sqrt(3)/2},-1/2) arc (-30-90:8.21+90-360:{sqrt(3)/5});
\draw ({4*sqrt(3)/7},1/7) arc (8.21-90:30+90-360:{sqrt(3)/9});
\draw ({sqrt(3)/2},1/2) arc (30-90:51.79+90-360:{sqrt(3)/9});
\draw ({5*sqrt(3)/14},11/14) arc (51.79-90:90+90-360:{sqrt(3)/5});

\end{tikzpicture}

\begin{tikzpicture}[scale=2]

\draw[thick] (-2.6, 0) -- (2.6, 0);

\draw[ultra thick] (-2.6,1) -- (2.6,1); 
\draw[ultra thick] (-2,0) node [below] {$-2 = \Phi^{-1} (\frac{-2}7,\frac 37, \frac 67)$} 
arc (-90:270:1/2); 
\draw[ultra thick] (-1,0) node [below = 15pt] {$-1 = \Phi^{-1} (\frac{-1}3,\frac 23, \frac 23)$} arc (-90:270:1/2); 
\draw[ultra thick] (0,0) node [below] {$0 = \Phi^{-1} (\frac 01,\frac 11,\frac 01)$} arc (-90:270:1/2); 
\draw[ultra thick] (1,0) node [below = 15pt] {$1 = \Phi^{-1} (\frac 11,\frac 01, \frac 01)$}  arc (-90:270:1/2); 
\draw[ultra thick] (2,0) node [below] {$2 = \Phi^{-1} (\frac 23, \frac{-1}3, \frac23)$} 
arc (-90:270:1/2); 

\draw (-2,1.2) -- (-2,0);
\draw (-1,1.2) -- (-1,0);
\draw (0,1.2) -- (0,0);
\draw (1,1.2) -- (1,0);
\draw (2,1.2) -- (2,0);

\draw[ultra thick] (-1/2,0) %node [below] {$(\frac{1}{2},\frac 12)$} 
arc (-90:270:1/8); 
\draw[ultra thick] (1/2,0) %node [below] {$(-\frac 12,\frac{3}{2})$} 
arc (-90:270:1/8); 

\draw (-2,0) arc (180:0:1/2);
\draw (-1,0) arc (180:0:1/2);
\draw (0,0) arc (180:0:1/2);
\draw (1,0) arc (180:0:1/2);

\draw[ultra thick] (1/3,0) node [below] {%$\frac{1}{3}$
} arc (-90:270:1/18); 
\draw[ultra thick] (2/3,0) node [below] {%$\frac{2}{3}$
} arc (-90:270:1/18); 

\draw (-1,0) arc (180:0:1/4);
\draw (-1/2,0) arc (180:0:1/4);
\draw (0,0) arc (180:0:1/4);
\draw (1/2,0) arc (180:0:1/4);

\draw (0,0) arc (180:0:1/6);
\draw (1/3,0) arc (180:0:1/12);
\draw (1/2,0) arc (180:0:1/12);
\draw (2/3,0) arc (180:0:1/6);

\end{tikzpicture}

\caption{Horocycles for the circle $\mathbf S^1_\III$. 
The radius of the horocycle at $(\frac ad, \frac bd, \frac cd)$ is $\sqrt 2/(\sqrt 3+2 d)$ on $\mathbf S^1_\III$ (above).
The projected horocycle at $\Phi^{-1}(\frac ad, \frac bd, \frac cd)$ on $\mathbb R$ has radius $1/(2(a+b))$ (below).
}\label{fig:S1III}
\end{figure}

A rational point $\mathbf z = (\frac a{a+b+c}, \frac b{a+b+c}, \frac c{a+b+c})$ on $\mathbf S^1_\III$ has height $\HH(\mathbf z) = a+b+c$, the Ford horocycle based at $\mathbf z$ has radius $$\rho = \frac{\sqrt 2}{\sqrt 3 + 2 (a+b+c)}.$$
Since $\mathbf c = (\frac 13, \frac 13, \frac 13)$, for $\mathbf z = (\frac a{a+b+c}, \frac b{a+b+c}, \frac c{a+b+c})$ and $\mathbf z' = (\frac {a'}{a'+b'+c'}, \frac {b'}{a'+b'+c'}, \frac {c'}{a'+b'+c'})$ we have
\begin{align*}
(\mathbf z - \mathbf c)\cdot (\mathbf z' - \mathbf c) &= \mathbf z \cdot \mathbf z' - \mathbf z \cdot \mathbf c - \mathbf z' \cdot \mathbf c +\mathbf c \cdot \mathbf c \\
&= \frac{aa' + bb'+cc'}{(a+b+c)(a'+b'+c')} - \frac 13 - \frac 13 + \frac 13 \\ 
&= \frac{(a+b+c)(a'+b'+c')}{(a+b+c)(a'+b'+c')} - \frac{ab'+ac'+ba'+bc'+ca'+cb'}{(a+b+c)(a'+b'+c')} - \frac 13\\
&= \frac 23 - \frac{ab'+ac'+ba'+bc'+ca'+cb'}{(a+b+c)(a'+b'+c')}.
\end{align*}
The inequality \eqref{ineq:claim} implies that
$$ab'+ac'+ba'+bc'+ca'+cb' \ge 1.$$
So we conclude that two horocycles based at $\mathbf z = (\frac a{a+b+c}, \frac b{a+b+c}, \frac c{a+b+c})$ and $\mathbf z' = (\frac {a'}{a'+b'+c'}, \frac {b'}{a'+b'+c'}, \frac {c'}{a'+b'+c'})$ are tangent if and only if 
$ab'+ac'+ba'+bc'+ca'+cb' = 1$.
%See Figure~\ref{fig:S1III} for some horocycles on $\mathbf B$ and projected horocycles on $\mathbb H^2$.
In Figure~\ref{fig:S1III}, we give pictures of horospheres based at the following rational points on $\mathbf S^1_\III$ and corresponding points on $\mathbb R$:
\begin{align*}
&\left( \tfrac {-2}7,\tfrac 37, \tfrac 67 \right)=\Phi\left(-2 \right), &
&  \left( \tfrac {-1}3,\tfrac 23, \tfrac 23 \right)=\Phi\left(-1 \right), &
& \left(\tfrac {-2}7, \tfrac 67, \tfrac 37 \right)=\Phi\left( -\tfrac{1}{2}\right), \\
&  \left( \tfrac {0}1,\tfrac 11, \tfrac 01 \right)=\Phi\left( 0 \right), &
& \left(\tfrac {3}7, \tfrac 67, \tfrac {-2}7\right)=\Phi\left( \tfrac{1}{3}\right), &
& \left(\tfrac 23, \tfrac 23, \tfrac {-1}3\right)=\Phi\left( \tfrac{1}{2}\right) , \\
& \left(\tfrac 67, \tfrac 37, \tfrac {-2}7\right)=\Phi\left( \tfrac{2}{3}\right), &
&  \left( \tfrac {1}1,\tfrac 01, \tfrac 01 \right)=\Phi\left( 1 \right), &
&  \left( \tfrac 23,\tfrac {-1}3, \tfrac 23 \right)=\Phi\left( 2 \right), \\
& \left( \tfrac 01,\tfrac 01, \tfrac 11\right)=\Phi(\infty).
\end{align*}
Note that the height of $\Phi^{-1} \left( \frac a{a+b+c}, \frac b{a+b+c}, \frac c{a+b+c} \right)$ is $a+b$.

%-------------------------

%Radius = 1

%1-dimension :

%$(\mathbf S, \HH_{\mathbf S})$ - $(\mathbf P, \HH^{\Psi}_{\mathbf P})$ - $(\mathbb R,  \sqrt 2\mathbb Q )$

%2-dimension :

%$(\mathbf S, \HH_{\mathbf S})$ - $(\mathbf P, \HH^{\Psi}_{\mathbf P})$ - $(\mathbb C, (1+i) \mathbb Q (i) ) = (\mathbb C, \mathbb Q (i))$ 

%----------------------

%Radius = $\sqrt 2$

%1-dimension :

%$(\mathbf S, \HH_{\mathbf S})$ - $(\mathbf P, \HH^{\Psi}_{\mathbf P})$ - $(\mathbb R,  \sqrt 2\mathbb Q )$  ??

%2-dimension :

%$(\mathbf S, \HH_{\mathbf S})$ - $(\mathbf P, \HH^{\Psi}_{\mathbf P})$ - $(\mathbb C, \mathbb Q (\sqrt 2 i) ) $ ?? 

%----------------------

%Slanted cone 

%1-dimension :

%$(\mathbf S, \HH_{\mathbf S})$ - $(\mathbf P, \HH^{\Psi}_{\mathbf P})$ - $(\mathbb R, \mathbb Q )$  ??

%2-dimension :

%$(\mathbf S, \HH_{\mathbf S})$ - $(\mathbf P, \HH^{\Psi}_{\mathbf P})$ - $(\mathbb C, \mathbb Q (\omega) ) $ ?? 

%----------------------

%Then, we have %the following relation  
%\begin{equation*}
%\frac{\HH_{\mathbf S}(\bm\phi(\frac xy))}{\HH_2(\frac xy)}
%= 1 + \left( \frac xy \right)^2.
%\end{equation*}
%Since 
%$$
%%\frac{ds}{dt} = 
%\left\| \frac{d\bm\phi(t)}{dt} \right\| = \frac{2}{t^2+1},
%$$
%we have 
%\begin{align*}
%\frac{\HH_{\mathbf S} \left( \bm\phi( \frac xy) \right)  \left\| \bm\phi(t) - \bm\phi( \frac xy)  \right\|} { \HH_2 \left(\frac xy\right) \left| t - \frac{p}{q} \right|}
%\to 2 %\frac{c}{q^2} \cdot \frac{2}{t^2 +1}
%%2 L_{S^1}(\bm\phi(t)).
%\end{align*}
%as $\frac xy \to t$.
%Therefore, we have %for $(\alpha,\beta) = \bm\phi(t)$
%$$L_2(t) = 2 L_{S^1} (\bm\phi(t)).$$

\subsection{The case $\mathbf{S} = \mathbf{S}_\I^2$}

\begin{figure}
\begin{tikzpicture}[scale=.07]
%	\draw[thick] (-2, 0) -- (5.6, 0);
%	\draw[->] (0, -1) -- (0, 4) node[above] {$\mathbf e$};

%	\draw (0,1.5) circle (1.732); 
	%node [above left] {$\mathbf n$}; 
%	\draw[thick] (0,.5) circle (1); %node [Right] {$\mathbf S$}; 

	\draw[thick] (1.5,-34.5)--(48,-55.5)--(83,-54)--(55,-2)--(1.5,-34.5)--(33.5,-83.5)--(83,-54);
	\draw[thick] (55,-2) --(48,-55.5)--(33.5,-83.5);
	
	\draw[dashed] (1.5,-34.5)--(41,-37)--(83,-54);
	\draw[dashed] (55,-2) --(41,-37)--(33.5,-83.5);
	
	\draw[fill] (1.5,-34.5) circle (0.70711) node [left] {$(\frac 11,\frac 01,\frac 01)$};
	\draw[fill] (48,-55.5) circle (0.70711) node [below right] {$(\frac 01,\frac 11,\frac 01)$};
	\draw[fill] (83,-54) circle (0.70711) node [right] {$(\frac{-1}1,\frac 01, \frac 01)$};
	\draw[fill] (41,-37) circle (0.70711) node [above left] {$(\frac 01, \frac{-1}1,\frac 01)$};
	\draw[fill] (55,-2) circle (0.70711) node [right] {$(\frac 01,\frac 01,\frac 11)$};
	\draw[fill] (33.5,-83.5) circle (0.70711) node [below] {$(\frac 01,\frac 01,\frac{-1}1)$};
\end{tikzpicture}

\begin{tikzpicture}[scale=4.5]
\draw[fill] (0,0) circle (0.01) node[left] {$0 = \Phi^{-1}(\frac 01,\frac{-1}1,\frac 01)$};   %%  1/2
\draw[fill] (1,0) circle (0.01) node[right] {$1 = \Phi^{-1}(\frac11,\frac01,\frac01)$};   %%  1/2
\draw[fill] (0,1) circle (0.01) node[left] {$\mathrm{i} = \Phi^{-1} (\frac{-1}1,\frac01,\frac01)$};   %%  1/2
\draw[fill] (1,1) circle (0.01) node[right] {$1+\mathrm{i} = \Phi^{-1}(\frac01,\frac11,\frac01)$};  

\draw[fill] (.5,.5) circle (0.01) node[right] {$\frac{1}{1-\mathrm{i}} = \Phi^{-1}(\frac01,\frac01,\frac{-1}1)$};

\draw (0,0) -- (1,0) -- (1,1) -- (0,1) -- (0,0);
\draw (0,0) -- (1,1);
\draw (0,1) -- (1,0);
\end{tikzpicture}

\caption{
These graphs show how horospheres at $\mathbf{S}$ or $\mathbb{C}$ are tangent to one another. %. In the upper graph, a vertex shows where a horosphere is tangent to $\mathbf{S}$ and, 
When two vertices are connected, the corresponding horosphere are tangent to one another. 
The radius of horospheres based at $(\frac ad, \frac bd, \frac cd)$ and $\Phi^{-1}(\frac ad, \frac bd, \frac cd)$ are $1/(1+\sqrt2 c)$ (above) and $1/(2(d-c))$ (below) respectively.
}\label{fig:S2I}
\end{figure}

A rational point $\mathbf z = (\frac ad, \frac bd, \frac cd)$ on $\mathbf S^2_\I$ has height $\HH(\mathbf z) = d$, the Ford horosphere based at $\mathbf z$ has radius $1/(1+\sqrt 2 d)$.
For  $\mathbf z = (\frac ad, \frac bd, \frac cd)$ and $\mathbf z' = (\frac {a'}{d'}, \frac {b'}{d'}, \frac {c'}{d'})$,
the inequality \eqref{ineq:claim} implies that 
\begin{align*}
(\mathbf z - \mathbf c)\cdot (\mathbf z' - \mathbf c)
&= \frac{aa' + bb'+cc'}{dd'} 
\le 1 - \frac{1}{dd'}.
\end{align*}
Therefore, we have
$aa' + bb' + cc' - dd'\le -1$
and two horospheres based at $\mathbf z = (\frac ad, \frac bd, \frac cd)$ and $\mathbf z' = (\frac {a'}{d'}, \frac {b'}{d'}, \frac {c'}{d'})$ are tangent if and only if 
$aa' + bb' + cc' - dd'= -1$.
In Figure~\ref{fig:S2I}, we give diagrams of horospheres based at six rational points on $\mathbf S^2_\I$ and corresponding points on $\mathbb C$:
\begin{align*}
&\left(\tfrac11,\tfrac01,\tfrac01\right) = \Phi(1), & &\left(\tfrac01,\tfrac11,\tfrac01\right) = \Phi(1+\mathrm{i}), & &\left(\tfrac01,\tfrac01,\tfrac11\right) = \Phi(\infty),\\
&\left(\tfrac{-1}1,\tfrac01,\tfrac 01\right) = \Phi(\mathrm{i}),& &\left(\tfrac 01,\tfrac{-1}1,\tfrac 01\right)= \Phi(0), & &\left(\tfrac01,\tfrac01,\tfrac{-1}1\right)= \Phi \left(\tfrac{1}{1-\mathrm{i}}\right).
\end{align*}
Note that the height of $\Phi^{-1} \left( \frac a{d}, \frac b{d}, \frac c{d} \right)$ is $d-c$.

%\begin{tabular}{ccc}
%$\mathbf S^2_\I$ & $\mathbf P_\I$ & $\mathbb C$ \\ \hline
%$(0,0,1)$ & $\infty$ & $\infty$ \\
%$(0,-1,0)$ & $(0,-1,0)$ & $0$ \\
%$(1,0,0)$ & $(1,0,0)$ & $1$ \\
%$(-1,0,0)$ & $(-1,0,0)$ & $\mathrm{i}$ \\
%$(0,1,0)$ & $(0,1,0)$ & $1+\mathrm{i}$ \\
%$(0,0,-1)$ & $(0,0,0)$ & $\frac {1}{1-\mathrm{i}}$ \\ \hline
%\end{tabular}

\subsection{The case $\mathbf{S} = \mathbf{S}_\II^2$}

\begin{figure}
\begin{tikzpicture}[scale=.08]
	\draw[thick] (1.5,-34.5)--(28.5,-54.5)--(76,-61.5)--(81.5,-47.5);
	\draw[thick] (1.5,-34.5)--(29,-0)--(60.5,-14.5)--(76,-61.5)--(57,-80)--(16.5,-75.5)--(1.5,-34.5);
	\draw[thick] (16.5,-75.5)--(28.5,-54.5)--(60.5,-14.5)--(69,-9);
	\draw[thick] (29,-0)--(69,-9)--(81.5,-47.5);
	
	\draw[dashed] (1.5,-34.5)--(18.5,-26)--(53.5,-33.5)--(81.5,-47.5);
	\draw[dashed] (29,-0)--(18.5,-26)--(31.5,-60.5)--(57,-80);
	\draw[dashed] (69,-9)--(53.5,-33.5)--(31.5,-60.5)--(16.5,-75.5);
	\draw[dashed] (57,-80)--(81.5,-47.5);
	
	\draw[fill] (1.5,-34.5) circle (0.70711) node [left] {$(\frac01,\frac{-1}1,\frac11)$};
	\draw[fill] (28.5,-54.5) circle (0.70711) node [left=2pt,yshift=-2pt] {$(\frac11,\frac01,\frac11)$};
	\draw[fill] (76,-61.5) circle (0.70711) node [right] {$(\frac11,\frac11,\frac01)$};
	\draw[fill] (81.5,-47.5) circle (0.70711) node [right] {$(\frac01,\frac11,\frac{-1}1)$};
	\draw[fill] (53.5,-33.5) circle (0.70711) node [below=10pt,xshift=9pt] {$(\frac{-1}1,\frac01,\frac{-1}1)$};
	\draw[fill] (18.5,-26) circle (0.70711) node [above right, yshift=-4pt] {$(\frac{-1}1,\frac{-1}1,\frac01)$};
	
	\draw[fill] (29,-0) circle (0.70711) node [above] {$(\frac{-1}1,\frac01,\frac11)$};
	\draw[fill] (60.5,-14.5) circle (0.70711) node [left=3pt, yshift=-3pt] {$(\frac01,\frac11,\frac11)$};
	\draw[fill] (69,-9) circle (0.70711) node [right] {$(\frac{-1}1,\frac11,\frac01)$};

	\draw[fill] (16.5,-75.5) circle (0.70711) node [below] {$(\frac11,\frac{-1}1,\frac01)$};
	\draw[fill] (31.5,-60.5) circle (0.70711) node [right=3pt,yshift=-2pt] {$(\frac01,\frac{-1}1,\frac{-1}1)$};
	\draw[fill] (57,-80) circle (0.70711) node [below] {$(\frac11,\frac01,\frac{-1}1)$};
\end{tikzpicture}

\begin{tikzpicture}[scale=5]
	\draw[fill] (0,0) circle (0.01) node[left] {$\frac 01 = \Phi^{-1} (\frac11,\frac01,\frac11)$};   %%  1/sq(2)
	\draw[fill] (1,0) circle (0.01) node[right] {$\frac 11 = \Phi^{-1} (\frac11,\frac11,\frac01)$};   %%  1/sq(2)
	\draw[fill] (0,{sqrt(2)}) circle (0.01) node[left] {$\omega = \Phi^{-1} (\frac{-1}1,\frac01,\frac11)$};   %%  1/sq(2)
	\draw[fill] (1,{sqrt(2)}) circle (0.01) node[right] {$1+\omega = \Phi^{-1} (\frac{-1}1,\frac11,\frac01)$};   %%  1/sq(2)

	\draw[fill] (0,{1/sqrt(2)}) circle (0.01) node[left] {$\frac{-1}{\omega}=\Phi^{-1} (\frac01,\frac{-1}1,\frac11)$};   %%  1/2sq(2)
	\draw[fill] (1,{1/sqrt(2)}) circle (0.01) node[right] {$\frac{-1+\omega}{\omega} = \Phi^{-1} (\frac01,\frac11,\frac{-1}1)$};   %%  1/2sq(2)

	\draw[fill] (1/3,{sqrt(2)/3}) circle (0.01) node[left=1pt,yshift=-2pt] {$\frac{1}{1-\omega}= \Phi^{-1} (\frac11,\frac{-1}1,\frac01)$};   %%  1/3sq(2)
	\draw[fill] (2/3,{sqrt(2)/3}) circle (0.01) node[right] {$\frac{\omega}{1+\omega} = \Phi^{-1} (\frac11,\frac01,\frac{-1}1)$};    %%  1/3sq(2)
	\draw[fill] (1/3,{2*sqrt(2)/3}) circle (0.01) node[left=1pt,yshift=2pt] {$\frac{-1+\omega}{1+\omega} =\Phi^{-1} (\frac{-1}1,\frac{-1}1,\frac01)$};   %%  1/3sq(2)
	\draw[fill] (2/3,{2*sqrt(2)/3}) circle (0.01) node[right,yshift=3pt] {$\frac{2}{1-\omega} = \Phi^{-1} (\frac{-1}1,\frac01,\frac{-1}1)$};    %%  1/3sq(2)

	\draw[fill] (1/2,{sqrt(2)/2}) circle (0.01) node[above,xshift=0pt] {$\frac{1+\omega}{2} = \Phi^{-1} (\frac01,\frac{-1}1,\frac{-1}1)$};    %%  1/4sq(2)

	\draw (0,0)--(1,0);
	\draw (0,0)--(0,{sqrt(2)});
	\draw (0,{sqrt(2)})--(1,{sqrt(2)});
	\draw (1,0)--(1,{sqrt(2)});
	\draw (0,0)--(1,{sqrt(2)});
	\draw (1,0)--(0,{sqrt(2)});
	\draw (1/3,{sqrt(2)/3})--(2/3,{sqrt(2)/3});
	\draw (1/3,{2*sqrt(2)/3})--(2/3,{2*sqrt(2)/3});
	\draw (0,{sqrt(2)/2})--(1/3,{sqrt(2)/3});
	\draw (0,{sqrt(2)/2})--(1/3,{2*sqrt(2)/3});
	\draw (1,{sqrt(2)/2})--(2/3,{sqrt(2)/3});
	\draw (1,{sqrt(2)/2})--(2/3,{2*sqrt(2)/3});
\end{tikzpicture}

\caption{These graphs show how horospheres at $\mathbf{S}$ or $\mathbb{C}$ are tangent to one another. 
When two vertices are connected, the corresponding horosphere are tangent to one another.
The radius of horospheres based at $(\frac ad, \frac bd, \frac cd)$ and $\Phi^{-1}(\frac ad, \frac bd, \frac cd)$ are ${\sqrt 2}/(1+2 c)$ (above) and $1/(2(2d-b-c))$ (below) respectively. 
}\label{fig:S2II}
\end{figure}

A rational point $\mathbf z = (\frac ad, \frac bd, \frac cd)$ on $\mathbf S^2_\II$ has height $\HH(\mathbf z) = d$, the Ford horosphere based at $\mathbf z$ has radius $\sqrt 2/(1+ 2 d)$.
For  $\mathbf z = (\frac ad, \frac bd, \frac cd)$ and $\mathbf z' = (\frac {a'}{d'}, \frac {b'}{d'}, \frac {c'}{d'})$,
the inequality \eqref{ineq:claim} implies that 
\begin{align*}
(\mathbf z - \mathbf c)\cdot (\mathbf z' - \mathbf c)
&= \frac{aa' + bb'+cc'}{dd'} 
\le 2 - \frac{1}{dd'}.
\end{align*}
Therefore, we have
$aa' + bb' + cc' - 2dd'\le -1$
and two horoshperes based at $\mathbf z = (\frac ad, \frac bd, \frac cd)$ and $\mathbf z' = (\frac {a'}{d'}, \frac {b'}{d'}, \frac {c'}{d'})$ are tangent if and only if 
$aa' + bb' + cc' - 2dd'= -1$.
In Figure~\ref{fig:S2II}, we give diagrams of horospheres based at twelve rational points on $\mathbf S^2_\II$ and corresponding points on $\mathbb C$: 
\begin{align*}
&\left(\tfrac11,\tfrac11,\tfrac01\right)  = \Phi( 1), & &
\left(\tfrac11,\tfrac{-1}1,\tfrac01\right)  = \Phi\left( \tfrac {1}{1-\omega}\right), & &
\left(\tfrac{-1}1,\tfrac11,\tfrac01\right)  = \Phi( 1+\omega), \\
&\left(\tfrac{-1}1,\tfrac{-1}1,\tfrac01\right) = \Phi\left(\tfrac {-1+\omega}{1+\omega}\right),& &
\left(\tfrac11,\tfrac01,\tfrac11\right)  = \Phi( 0), & &
\left(\tfrac11,\tfrac01,\tfrac{-1}1\right)  = \Phi\left( \tfrac \omega{1+\omega}\right), \\ 
&\left(\tfrac{-1}1,\tfrac01,\tfrac11\right)  = \Phi( \omega), & &
\left(\tfrac{-1}1,\tfrac01,\tfrac{-1}1\right) = \Phi\left( \tfrac {2}{1-\omega}\right) & &
\left(\tfrac01,\tfrac11,\tfrac11\right) = \Phi( \infty), \\
&\left(\tfrac01,\tfrac11,\tfrac{-1}1\right)  = \Phi\left( \tfrac {-1+\omega}{\omega}\right), & &
\left(\tfrac01,\tfrac{-1}1,\tfrac11\right)  = \Phi\left( -\tfrac 1{\omega}\right), & &
\left(\tfrac01,\tfrac{-1}1,\tfrac{-1}1\right) = \Phi\left( \tfrac {1+\omega}{2}\right).
\end{align*}
Note that the height of $\Phi^{-1} \left( \frac a{d}, \frac b{d}, \frac c{d} \right)$ is $2d-b-c$.

%\begin{tabular}{ccc}
%$\mathbf S^2_\II$ & $\mathbf P_\II$ & $\mathbb C$ \\ \hline
%$(0,1,1)$ & $\infty$ & $\infty$ \\
%$(1,0,0)$ & $(1,0,1)$ & $0$ \\
%$(1,1,0)$ & $(1,1,0)$ & $1$ \\
%$(-1,0,1)$ & $(-1,0,1)$ & $\omega$ \\
%$(-1,1,0)$ & $(-1,1,0)$ & $1+\omega$ \\
%$(0,-1,1)$ & $(0,0,1)$ & $-\frac 1{\omega}$ \\
%$(0,1,-1)$ & $(0,1,0)$ & $\frac {-1+\omega}{\omega}$ \\
%$(1,-1,0)$ & $(\frac 13,\frac 13,\frac 23)$ & $\frac {1}{1-\omega}$ \\
%$(1,0,-1)$ & $(\frac 13,\frac 23,\frac 13)$ & $\frac \omega{1+\omega}$ \\ 
%$(-1,-1,0)$ & $(-\frac 13,\frac 13,\frac 23)$ & $\frac {-1+\omega}{1+\omega}$ \\
%$(-1,0,-1)$ & $(-\frac 13,\frac 23,\frac 13)$ & $\frac {2}{1-\omega}$ \\
%$(0,-1,-1)$ & $(0,\frac 12, \frac 12)$ & $\frac {1+\omega}{2}$ \\ \hline
%\end{tabular}

\subsection{The case $\mathbf{S} = \mathbf{S}_\III^2$}

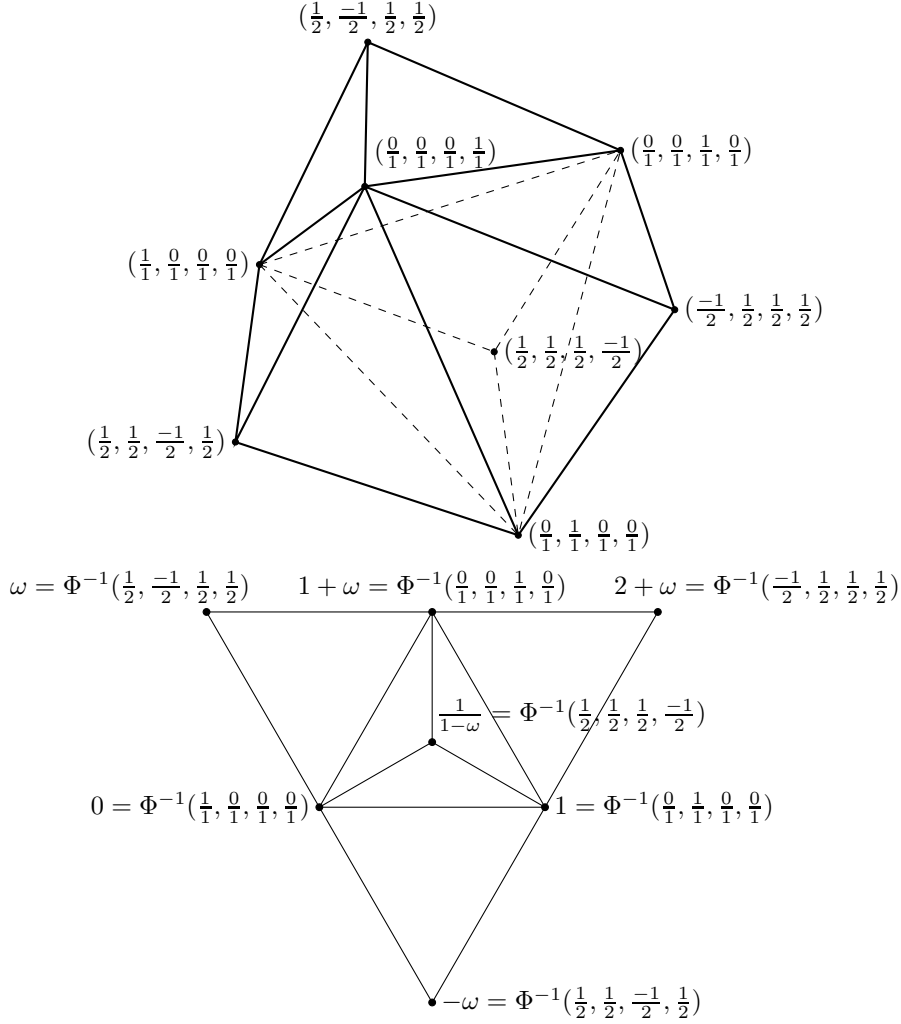
\begin{figure}
\begin{tikzpicture}[scale=.08]
%	\draw[thick] (-2, 0) -- (5.6, 0);
%	\draw[->] (0, -1) -- (0, 4) node[above] {$\mathbf e$};

%	\draw (0,1.5) circle (1.732); 
	%node [above left] {$\mathbf n$}; 
%	\draw[thick] (0,.5) circle (1); %node [Right] {$\mathbf S$}; 

	\draw[dashed] (4.5,-37.5)--(47.5,-82.5)--(64.5,-18.5);
	\draw[thick] (64.5,-18.5)--(22,-24.5)--(4.5,-37.5);
	\draw[thick] (22,-24.5)--(47.5,-82.5);
	
	\draw[dashed] (4.5,-37.5)--(64.5,-18.5);
	
	\draw[thick] (22,-24.5)--(.5,-67)--(4.5,-37.5);
	\draw[thick] (.5,-67)--(47.5,-82.5);

	\draw[thick] (22,-24.5)--(73.5,-45)--(64.5,-18.5);
	\draw[thick] (73.5,-45)--(47.5,-82.5);

	\draw[dashed] (4.5,-37.5)--(43.5,-52)--(64.5,-18.5);
	\draw[dashed] (43.5,-52)--(47.5,-82.5);

	\draw[thick] (4.5,-37.5)--(22.5,-.5)--(64.5,-18.5);
	\draw[thick] (22.5,-.5)--(22,-24.5);

	\draw[fill] (22,-24.5) circle (0.5) node [right,yshift=13pt] {$(\frac 01,\frac 01,\frac 01,\frac 11)$};
	\draw[fill] (4.5,-37.5) circle (0.5) node [left] {$(\frac 11,\frac01,\frac01,\frac01)$};
	\draw[fill] (47.5,-82.5) circle (0.5) node [right] {$(\frac01,\frac11,\frac01,\frac01)$};
	\draw[fill] (64.5,-18.5) circle (0.5) node [right] {$(\frac01,\frac01,\frac11,\frac01)$};
	
	\draw[fill] (.5,-67) circle (0.5) node [left] {$(\frac 12,\frac 12,\frac{-1}2,\frac 12)$};
	\draw[fill] (73.5,-45) circle (0.5) node [right] {$(\frac{-1}2,\frac 12,\frac 12,\frac 12)$};
	\draw[fill] (43.5,-52) circle (0.5) node [right] {$(\frac 12,\frac 12,\frac 12,\frac{-1}2)$};
	\draw[fill] (22.5,-.5) circle (0.5) node [above] {$(\frac 12,\frac{-1}2,\frac 12,\frac 12)$};
\end{tikzpicture}
\begin{tikzpicture}[scale=3]
%	\draw[->] (-.6, 0) -- (2, 0) node[right] {$x$};
%	\draw[->] (0, -.6) -- (0, 1.5) node[above] {$y$};
	\draw[fill] (0,0) circle (0.015) node[left] {$0 = \Phi^{-1}(\frac11,\frac01,\frac01,\frac01)$};   %%  1/2
	\draw[fill] (1,0) circle (0.015) node[right] {$1 = \Phi^{-1}(\frac01,\frac11,\frac01,\frac01)$};   %%  1/2
	\draw[fill] (.5,.866) circle (0.015) node[above] {$1+\omega = \Phi^{-1}(\frac01,\frac01,\frac11,\frac01)$};   %%  1/2
	\draw[fill] (-.5,.866) circle (0.015) node[above left, xshift=20pt] {$\omega = \Phi^{-1}(\frac 12, \frac{-1}2, \frac 12, \frac 12)$};   %%  1/2
	\draw[fill] (1.5,.866) circle (0.015) node[above right, xshift=-20pt] {$2+\omega = \Phi^{-1}(\frac{-1}2, \frac 12, \frac 12, \frac 12)$};   %%  1/2
	\draw[fill] (.5,-.866) circle (0.015) node[right] {$-\omega =\Phi^{-1}(\frac 12, \frac 12, \frac{-1}2, \frac 12)$};   %%  1/2
	\draw[fill] (.5,.2887) circle (0.015) node[above right,xshift=-2pt] {$\frac{1}{1-\omega}= \Phi^{-1}(\frac 12, \frac 12, \frac 12, \frac{-1}2)$};   %%  1/6

%	\draw[fill] (0,.57735) circle (0.01) node[below] {$\frac{\omega}{2-\omega}$};   %%  1/6
%	\draw[fill] (1,.57735) circle (0.01) node[below] {$\frac{2}{2-\omega}$};   %%  1/6

	\draw (0,0) -- (1,0)--(.5,.866)--(0,0);
    \draw (-.5,.866)--(.5,-.866)--(1.5,.866)--(-.5,.866);
	\draw (1/2,.2887)--(1/2,.866);
	\draw (1/2,.2887) -- (0,0);
	\draw (1/2,.2887) -- (1,0);
	
\end{tikzpicture}
\caption{These graphs show how horospheres at $\mathbf{S}$ or $\mathbb{C}$ are tangent to one another.  
When two vertices are connected, the corresponding horosphere are tangent to one another.
The radius of horospheres based at $\mathbf z = (\frac a{a+b+c+d}, \frac b{a+b+c+d}, \frac c{a+b+c+d}, \frac d{a+b+c+d})$ and $\Phi^{-1}(\mathbf z)$ are ${\sqrt 3}/(2+\sqrt 6(a+b+c+d))$ (above) and $1/(2(a+b+c))$ (below) respectively. 
}\label{fig:S2III}
\end{figure}

A rational point $\mathbf z = \left(\frac{a}{a+b+c+d}, \frac{b}{a+b+c+d}, \frac{c}{a+b+c+d}, \frac{d}{a+b+c+d} \right)$ on $\mathbf S^2_\III$ has height $\HH(\mathbf z) = a+b+c+d$, the Ford horosphere based at $\mathbf z$ has radius 
$$\rho = \frac{\sqrt 3}{2 + \sqrt 6 (a+b+c+d)}.$$
Since $\mathbf c = (\frac 14, \frac 14, \frac 14, \frac 14)$, for two rational points $\mathbf z$ and $\mathbf z'$ we have
\begin{align*}
(\mathbf z - \mathbf c)\cdot (\mathbf z' - \mathbf c) &=
\mathbf z \cdot \mathbf z' - \mathbf z \cdot \mathbf c - \mathbf z' \cdot \mathbf c + \mathbf c \cdot \mathbf c \\
&= \frac{aa'+bb'+cc'+dd'}{(a+b+c+d)(a'+b'+c'+d')} - \frac 14 - \frac 14 + \frac 14 \\
&= \frac{(a+b+c+d)(a'+b'+c'+d')}{(a+b+c+d)(a'+b'+c'+d')} - \frac 14 \\
&\quad - \frac{ab'+ac'+ad'+ba'+bc'+bd'+ca'+cb'+cd'+da'+db'+dc'}{(a+b+c+d)(a'+b'+c'+d')}\\
&= \frac 34 - \frac{ab'+ac'+ad'+ba'+bc'+bd'+ca'+cb'+cd'+da'+db'+dc'}{(a+b+c+d)(a'+b'+c'+d')}.
\end{align*}
The inequality \eqref{ineq:claim} implies that
$$ab'+ac'+ad'+ba'+bc'+bd'+ca'+cb'+cd'+da'+db'+dc' \ge 1.$$
Two horospheres based at $\mathbf z$ and $\mathbf z'$ are tangent if and only if 
\[
ab'+ac'+ad'+ba'+bc'+bd'+ca'+cb'+cd'+da'+db'+dc' = 1.
\]
In Figure~\ref{fig:S2III}, we give diagrams of horospheres based at eight rational points on $\mathbf S^2_\III$ and corresponding points on $\mathbb C$: 
\begin{align*}
&\left(\tfrac11,\tfrac01,\tfrac01,\tfrac01\right)  = \Phi( 0), & &
\left(\tfrac01,\tfrac11,\tfrac01,\tfrac01\right) = \Phi( 1), & &
\left(\tfrac01,\tfrac01,\tfrac11,\tfrac01\right)  = \Phi( 1+\omega), \\
&\left(\tfrac01,\tfrac01,\tfrac01,\tfrac11\right)  = \Phi( \infty), & &
\left(-\tfrac 12, \tfrac 12, \tfrac 12, \tfrac 12\right)  = \Phi( 2+\omega), & & 
\left(\tfrac 12, -\tfrac 12, \tfrac 12, \tfrac 12\right)  = \Phi( \omega), \\
&\left(\tfrac 12, \tfrac 12, -\tfrac 12, \tfrac 12\right)  = \Phi( -\omega), & &
\left(\tfrac 12, \tfrac 12, \tfrac 12, -\tfrac 12\right)  = \Phi\left(\tfrac 1{1-\omega}\right) .
\end{align*}
Note that the height of $\Phi^{-1} \left( \frac a{a+b+c+d}, \frac b{a+b+c+d}, \frac c{a+b+c+d}, \frac d{a+b+c+d} \right)$ is $a+b+c$.

\section*{Acknowledgements}

The authors wish to thank the anonymous referee for the careful reading and helpful suggestions.
The research was supported by the National Research Foundation of Korea (NRF-2018R1A2B6001624).

\end{document}